\documentclass[12pt,reqno]{amsart}

\usepackage[margin=1.18in]{geometry}
\usepackage{amsmath,amssymb,mathtools,mathrsfs,bm}
\usepackage{booktabs,array,enumitem}
\usepackage{aliascnt}
\usepackage{xcolor}
\usepackage{comment}
\usepackage{microtype}
\usepackage{hyperref}
\usepackage{xurl}
\usepackage[capitalise,nameinlink]{cleveref}
\usepackage{embedfile}
\usepackage[backend=biber,style=alphabetic,sorting=nyt,doi=true,isbn=false,url=true,eprint=true,maxbibnames=99]{biblatex}
\usepackage{stmaryrd}
\renewbibmacro{in:}{}

\hypersetup{
  colorlinks=true,
  linkcolor=blue!55!black,
  citecolor=green!40!black,
  urlcolor=blue!60!black,
  pdftitle={Sharp partial regularity for Hamiltonian stationary Lagrangian graphs},
  pdfauthor={Arunima Bhattacharya, Gerard Orriols, Anna Skorobogatova}
}

\AtEveryBibitem{%
  \ifentrytype{article}{\clearfield{url}\clearfield{urldate}}{}
  \ifentrytype{book}{\clearfield{url}\clearfield{urldate}}{}
}

\newcommand{\R}{\mathbb R}
\newcommand{\C}{\mathbb C}

\newcommand{\Z}{\mathbb Z}
\newcommand{\cH}{\mathcal H}
\newcommand{\cQ}{\mathcal Q}
\newcommand{\eps}{\varepsilon}
\newcommand{\dd}{\mathrm d}
\newcommand{\vol}{\mathrm{vol}}
\newcommand{\Id}{\mathrm{Id}}
\newcommand{\Sing}{\operatorname{Sing}}
\newcommand{\Reg}{\operatorname{Reg}}
\newcommand{\graph}{\operatorname{graph}}
\newcommand{\Mass}{\mathbf M}
\newcommand{\FSL}{\mathcal F}
\newcommand{\Sym}{\operatorname{Sym}}

\newcommand{\spt}{\operatorname{spt}}

\newcommand{\Rea}{\operatorname{Re}}
\newcommand{\Ima}{\operatorname{Im}}
\newcommand{\tr}{\operatorname{tr}}
\newcommand{\Span}{\operatorname{span}}
\newcommand{\Lip}{\operatorname{Lip}}

\newcommand{\slot}{\mathbin{\lrcorner}}

\newtheorem{theorem}{Theorem}[section]
\newaliascnt{question}{theorem}
\newtheorem{question}[question]{Question}
\aliascntresetthe{question}
\newaliascnt{proposition}{theorem}
\newtheorem{proposition}[proposition]{Proposition}
\aliascntresetthe{proposition}
\newaliascnt{lemma}{theorem}
\newtheorem{lemma}[lemma]{Lemma}
\aliascntresetthe{lemma}
\newaliascnt{corollary}{theorem}
\newtheorem{corollary}[corollary]{Corollary}
\aliascntresetthe{corollary}
\newaliascnt{definition}{theorem}

\aliascntresetthe{definition}

\newaliascnt{remark}{theorem}
\theoremstyle{remark}
\newtheorem{remark}[remark]{Remark}
\aliascntresetthe{remark}
\theoremstyle{plain}

\crefname{theorem}{Theorem}{Theorems}
\crefname{question}{Question}{Questions}
\crefname{proposition}{Proposition}{Propositions}
\crefname{lemma}{Lemma}{Lemmas}
\crefname{corollary}{Corollary}{Corollaries}
\crefname{definition}{Definition}{Definitions}
\crefname{remark}{Remark}{Remarks}

\newcommand{\ttimes}{%
  \mathbin{\vcenter{\hbox{%
    \ooalign{%
      \hfil\raise.18ex\hbox{$\times$}\hfil\cr
      \hfil\lower.18ex\hbox{$\times$}\hfil\cr
    }%
  }}}%
}

\newcommand{\Fcal}{\mathcal{F}}
\newcommand{\Gcal}{\mathcal{G}}
\newcommand{\Hcal}{\mathcal{H}}

\newcommand{\Kcal}{\mathcal{K}}
\newcommand{\Lcal}{\mathcal{L}}

\newcommand{\Ucal}{\mathcal{U}}

\newcommand{\Oscr}{\mathscr{O}}

\newcommand{\Bbf}{\mathbf{B}}
\newcommand{\Cbf}{\mathbf{C}}

\newcommand{\Sbb}{\mathbb{S}}

\newcommand{\Zbb}{\mathbb{Z}}

\allowdisplaybreaks

\title[Sharp partial regularity]{Sharp partial regularity of Hamiltonian stationary and special Lagrangian graphs}

\author{Arunima Bhattacharya}
\address{Department of Mathematics, University of North Carolina at Chapel Hill}
\email{arunimab@unc.edu}

\author{Gerard Orriols}
\address{Department of Mathematics, University of Cambridge}
\email{go262@cam.ac.uk}

\author{Anna Skorobogatova}
\address{Institute for Theoretical Studies, ETH Z\"urich}
\email{anna.skorobogatova@eth-its.ethz.ch}

\begin{document}

\begin{abstract}
We prove a sharp partial regularity result for Hamiltonian stationary Lagrangian Lipschitz submanifolds in arbitrary smooth almost K\"ahler manifolds: every weak solution of the corresponding equation is smooth away from a relatively closed singular set of Hausdorff dimension at most $n-5$.  We show that the estimate is optimal by constructing a nonzero two-homogeneous viscosity solution
\[
  U\in C^{1,1}(\R^5)\setminus C^2(\R^5)
\]
of the phase-zero special Lagrangian equation, whose level sets on $\Sbb^4$ are the leaves of Cartan's isoparametric foliation.
Its gradient graph is a non-flat calibrated cone, real analytic away from the vertex. This also gives the first $C^{1,1}$ but non-$C^2$ solution of the special Lagrangian equation, and shows that the same dimensional estimate is sharp in the case of special Lagrangian graphs.
\end{abstract}

\maketitle
\tableofcontents

\section{Introduction}
Special Lagrangian submanifolds, introduced by Harvey--Lawson \cite{HL82}, are volume-minimizing submanifolds of a Calabi--Yau manifold and are central to the geometry underlying mirror symmetry. A converse of this fact is that minimal Lagrangian submanifolds are automatically special Lagrangian, and this leads naturally to the approach of constructing special Lagrangians by minimizing the area among Lagrangians. Submanifolds satisfying the corresponding first variation condition, known as Hamiltonian stationary Lagrangian submanifolds, were introduced in this variational form by Oh \cite{Oh1993}; they arise in questions of existence and stability in a fixed Hamiltonian isotopy class and have since been studied from various viewpoints.

The regularity theory of Hamiltonian stationary submanifolds has been studied by many over the past thirty years, yet it remains far less developed than the corresponding theory for minimal surfaces. Already in dimension two and in the minimizing case, non-flat conical singularities appear, as shown by Schoen--Wolfson \cite{SchoenWolfson2001}, and in higher dimensions no monotonicity formula is known, and therefore a tangent cone analysis is not available (see \cite{Orriols2026}). Thus, a refined understanding of singularities seems completely out of reach in general, and much of the recent progress has been restricted to the graphical case \cite{ChenWarren2019,BCW2023, BhattacharyaOgden2025}.

In this paper, we study the regularity of Hamiltonian stationary Lagrangian submanifolds
in almost K\"ahler manifolds which can be locally written as Lipschitz graphs. We show that every weak solution to the corresponding fourth order equation is smooth away from a relatively closed set of Hausdorff dimension at most $n-5$. 
We then show that this estimate is sharp, by constructing a five-dimensional non-flat graphical special Lagrangian cone in $\C^5$.
This therefore settles the question of the optimal dimensional estimate of the singular set for the fully nonlinear second-order special Lagrangian equation among $C^{1,1}$ functions.
\subsection{Main results}
Let $L^n$ be a Lagrangian submanifold in $\C^n = \R^n \times \R^n$ which is a Lipschitz graph over a simply connected domain $\mathcal{U} \subset \R^n$. It is well known that $L$ is then a gradient graph:
\[
  L = L_u = \{(x,Du(x)):x\in\Ucal\}
  \quad
  \text{for a function } u \in C^{1,1}(\Ucal).
\]
Recall that the Lagrangian angle $\beta$ of $L$ is defined by
\[
  \Omega|_{TL_u}=e^{i\beta}\vol_{L_u},
  \qquad
  \Omega=\dd z_1\wedge\cdots\wedge\dd z_n,
\]
and is given by
\begin{equation}\label{eq:Lag-angle}
  \beta=F_n(D^2u),
\end{equation}
where $F_n$ is the special Lagrangian operator
\begin{equation*}\label{eq:SL-operator}
  F_n(A):=\sum_{j=1}^n\arctan\lambda_j(A),
  \qquad A \in \Sym(n).
\end{equation*}
Here $\{ \lambda_j(A) \}$ are the eigenvalues of $A$ and the arctangent takes values in $(-\pi/2,\pi/2)$.

The mean-curvature identity of Harvey--Lawson gives $H=J \nabla_g\beta$. Thus a prescribed angle or phase $\theta_0:\Ucal\to(-n\pi/2,n\pi/2)$ leads to the second-order Lagrangian mean curvature equation
\begin{equation}\label{eq:sLag}
  F_n(D^2u)=\theta_0\,.
\end{equation}
When $\theta_0$ is constant, this is the special Lagrangian equation. The Hamiltonian stationary equation is instead the fourth-order equation
\begin{equation}\label{eq:HSL}
  \Delta_g\beta=0\,,
\end{equation}
with $\beta$ as in \eqref{eq:Lag-angle} and $ g_{ij}=\delta_{ij}+\partial_{ik}u\,\delta^{kl}\,\partial_{lj}u $ the induced metric.

It turns out that essentially the same framework can be set up locally for graphs in an arbitrary almost K\"ahler manifold $(M^{2n}, \omega, J, g)$. We will call $L \subset M$ a \emph{Lagrangian Lipschitz submanifold} if $L$ can be written, in local Darboux charts $\Psi : B_r \subset \C^n \to M$, as the graph of the gradient of a $C^{1,1}$ function:
\begin{equation}\label{eq:DW-graph}
  \Phi_u(x) := \Psi(x,Du(x))\,,
  \qquad
  L_u := \Phi_u(B_r)\,,
  \qquad
  L \cap \Psi(B_r) = L_u \cap \Psi(B_r)\,.
\end{equation}
Given such an $L_u$, one can locally assign to it a generalized Lagrangian angle $\beta(x)$, given by a fully nonlinear operator $\FSL(x, Du, D^2 u)$, and then the Hamiltonian stationarity condition is equivalent to \eqref{eq:HSL} with an additional divergence term for the pullback metric; see \cref{sec:lag-angle} for details.

In either case, we are led to consider $C^{1,1}_{\mathrm{loc}}$ weak solutions of the Hamiltonian-stationary system in the following sense: we assume
\begin{equation}
\label{eq:main-assumption}
  u\in C^{1,1}_{\mathrm{loc}}(\Ucal),
  \qquad
  \beta := \FSL(x, Du, D^2u) \in W^{1,2}_{\mathrm{loc}}(\Ucal),
\end{equation}
with $\FSL$ as in \eqref{eq:variable-phase-equation}, and
\begin{equation}
\label{eq:main-weak}
  \int_{\Ucal}\sqrt{\det g}\,g^{ij} (\partial_i \beta - b_i) \partial_j\varphi \,\dd x=0
  \qquad\text{for every }\varphi\in C_c^\infty(\Ucal).
\end{equation}
where $b$ is an $L^\infty$ vector field on $\mathcal{U}$ determined by the geometry ($b=0$ in the Calabi--Yau case).

Our first result is the following partial regularity theorem. In contrast to previous higher-dimensional results \cite{ChenWarren2019,BCW2023,BhattacharyaOgden2025}, it imposes no restriction on the phase range and requires neither convexity of the potential nor a smallness assumption on the Hessian.

\begin{theorem}\label{thm:regularity-main}
Suppose that $u\in C^{1,1}_{\mathrm{loc}}(\Ucal)$ satisfies \eqref{eq:main-assumption} and \eqref{eq:main-weak}, where $\Ucal\subset\R^n$ is open. Then there is a relatively closed set $\Sigma_u\subset\Ucal$ such that
\[
  u\in C^\infty(\Ucal\setminus\Sigma_u)
  \qquad \text{and} \qquad
  \dim_{\cH}\Sigma_u\le n-5.
\]
Moreover, for $n\le4$ the set $\Sigma_u$ is empty, and for $n=5$ it is locally finite.
\end{theorem}

As an immediate corollary, we have the following, which is merely an intrinsic re-statement of \cref{thm:regularity-main}. 

\begin{corollary}\label{cor:global-main}
Let $(M^{2n},\omega,J,g)$ be a smooth almost K\"ahler manifold and let $L$ be a Hamiltonian stationary Lagrangian Lipschitz submanifold with Lagrangian angle locally in $W^{1,2}$. Then there exists a relatively closed set $\Sing(L) \subset L$ such that
\[
  L \text{ is smooth outside } \Sing(L)
  \qquad \text{and} \qquad
  \dim_{\cH} \Sing(L) \le n-5.
\]
Moreover, for $n\le4$, the $\Sing(L)$ is empty, and for $n=5$ it is locally finite.
\end{corollary}

The full regularity statement for $n\le4$ was proved in \cite{BW19_2d,Bh25} in Euclidean space. Our contribution to the theorem above concerns the case $n \geq 5$ and its extension to manifolds. Note that the dimension bound was previously established in \cite[Theorem 3.9]{Dimler2023} in the particular case of Lipschitz special Lagrangian graphs in $\C^n$.

\medskip

The next natural question concerns the sharpness of the above result. If a $C^{1,1}$ Hamiltonian stationary potential were non-smooth, the geometric tangent object obtained as a blow-up at a given singularity would be a singular graphical special Lagrangian cone. This observation leads to the following rigidity question, which has been a long-standing open problem; see, for instance, the work of
Jost--Xin \cite{JostXin2002} and Yuan \cite{Yuan2002} as well as Mooney--Savin \cite[pg 2]{MooneySavin2024}.

\begin{question}\label{q:sing-sLag-cones}
Does there exist, in some complex dimension $n\ge5$, a non-flat multiplicity-one graphical special Lagrangian cone whose potential is $C^{1,1}$? Equivalently, does there exist a two-homogeneous function $U\in C^{1,1}(\R^n)$ that is not quadratic (hence $U\notin C^2(\R^n)$) and solves the special Lagrangian equation?
\end{question}

Here, we place emphasis on the fact that the cone must be \emph{graphical}. Indeed, in the non-graphical setting, there exist a wide range of known examples of singular special Lagrangian cones \cite{SchoenWolfson99,Haskins04,HaskinsKapouleas07,Ohnita,HaskinsKapouleas13}, and recent works including \cite{Pacini-I,Pacini-II,DimlerGaia26} further demonstrate that singularities of special Lagrangians can even be prescribed at isolated points. 

Our next result, perhaps unexpectedly, shows that the answer to the above question is in fact positive.

\begin{theorem}\label{thm:counterexample-main}
There exists a nonzero two-homogeneous function
\[
  U\in C^{1,1}(\R^5)\setminus C^2(\R^5)
\]
which solves the zero phase special Lagrangian equation $F_5(D^2U)=0$ in the viscosity sense on all of $\R^5$. Moreover, $U$ is real analytic on $\R^5\setminus\{0\}$ and its gradient graph is a non-flat multiplicity-one special Lagrangian cone calibrated by $\Rea(\dd z_1\wedge\cdots\wedge\dd z_5)$.
\end{theorem}

Note that taking products with additional euclidean factors yields examples in $\R^n$ for every $n \geq 5$ with an $(n-5)$-dimensional singular set; thus the dimensional bound in \cref{thm:regularity-main} (and \cref{cor:global-main}) is optimal.

\bigskip

In the minimal surface setting, the classical work of Barbosa \cite{Barbosa1979} and Fischer-Colbrie \cite{FischerColbrie1980}, addresses low-dimensional rigidity: every three-dimensional Lipschitz minimal graph is smooth, in arbitrary codimension. Through Federer’s dimension-reduction argument, this rigidity underlies the sharp $n-4$ bound for the singular set of area-minimizing Lipschitz graphs of arbitrary codimension; in fact, the same bound holds more generally for Lipschitz minimal graphs (see Dimler \cite{Dimler2023}). The bound is sharp even within the area-minimizing class, as witnessed by the calibrated Lawson--Osserman cone \cite{HL82}. In contrast to the second-order minimal surface system, the Hamiltonian stationary equation \eqref{eq:HSL} is fourth order, so no maximum principle is available directly for the potential.

Moreover, Hamiltonian stationary submanifolds do not have the monotonicity of mass ratios enjoyed by minimal submanifolds; this is already visible in complex dimension two (see the work of Minicozzi \cite{Minicozzi1995} and Schoen--Wolfson \cite{SchoenWolfson2001}), and higher-dimensional exact examples are discussed in the work of Orriols \cite[Appendix~B]{Orriols2026}. However, for a graph in $\C^n$, the equation decomposes into two second-order equations. When the Hessian of the potential is bounded, the Lagrangian angle satisfies a uniformly elliptic divergence-form equation, and the potential solves the uniformly elliptic Lagrangian mean curvature equation. This observation is behind the regularity results in \cite{ChenWarren2019,BW19_2d,Bh25} and is also the starting point here. However, to the best of our knowledge, this double second-order structure had not been exploited on manifolds, not even in the Calabi--Yau case. 

In general dimension, the earlier regularity results for single-valued Hamiltonian stationary graphs required additional hypotheses. Chen--Warren proved smoothness of $C^{1,1}$ weak solutions under a supercritical phase assumption, or a convexity assumption, or a small Hessian bound assumption, and also proved real analyticity for $C^1$ Hamiltonian stationary Lagrangian submanifolds in $\C^n$ \cite{ChenWarren2019}. In a symplectic manifold, Bhattacharya--Chen--Warren proved smoothness for a $C^1$ Hamiltonian stationary Lagrangian submanifold; the potential in the local gradient representation is assumed to be $C^2$ \cite{BCW2023}. In dimension two, all phase and convexity assumptions were removed in Bhattacharya--Warren \cite{BW19_2d}, and the same conclusion was obtained in dimensions three and four by Bhattacharya \cite{Bh25}. Below the Lipschitz graph threshold, Bhattacharya--Ogden proved smoothness for $C^{0,\gamma}$ Hamiltonian stationary graphs with $\gamma>1/3$ in the supercritical range and constructed singular examples when $\gamma=1/3$ \cite{BhattacharyaOgden2025}. Bhattacharya--Skorobogatova obtained a partial regularity theorem for a wider class of fourth order variational integrals on Hessian spaces \cite{BhattacharyaSkorobogatova2025}. These results do not decide what happens for a general bounded-Hessian Hamiltonian stationary graph when $n\ge5$.

There is a related distinction for the second-order equation \eqref{eq:sLag}. The values $\pm (n-2)\pi/2$ are referred to as the critical phases, since Yuan \cite{Yuan2006} showed that the level sets of the arctangent operator are convex whenever the absolute value of the phase is at least this large. In this critical and supercritical range, i.e., $|\theta_0|\geq (n-2)\frac{\pi}{2}$, interior Hessian estimates are known in every dimension \cite{WarrenYuan2009,WarrenYuan2010,WangYuan2014, Lcomp, RSsLag}; for variable phases, one needs additional regularity assumptions on the phase \cite{Bhattacharya2021,AB2d, Siyuan, BhattacharyaMooneyShankar2022,Zhou2025,Ding2024}. In the subcritical range, i.e., $|\theta_0|<(n-2)\frac{\pi}{2}$, this convexity is lost, and singular non-$C^{1,1}$ solutions are known \cite{NadirashviliVladut2010,WangYuan2013,MooneySavin2024}. Thus, the $C^{1,1}$ class lies at a natural borderline: the Hessian is bounded and the equation is uniformly elliptic along the solution, but the level set need not be convex.

We finally mention some other developments in the variational theory. Oh identified the Hamiltonian stability operator for minimal Lagrangians \cite{Oh1990}. Schoen--Wolfson studied area minimization among Lagrangian surfaces and the formation of isolated conical singularities \cite{SchoenWolfson2001}, and Joyce--Lee--Schoen constructed small Hamiltonian stationary Lagrangians in compact symplectic manifolds from rigid Euclidean models \cite{JoyceLeeSchoen2009}. Compactness under volume and total curvature bounds was proved in $\C^n$ by Chen--Warren and in compact symplectic manifolds by Chen--Ma \cite{ChenWarrenCompact2024,ChenMa2024}. Chen--Warren also introduced a volume-decreasing flow within a Hamiltonian isotopy class \cite{ChenWarren2024}. In the weak and parametrized setting, Pigati--Rivi\`ere proved regularity away from a locally finite singular set for parametrized Hamiltonian stationary Legendrian varifolds \cite{PigatiRiviere2024}; Gaia--Orriols--Rivi\`ere constructed Hamiltonian stationary surfaces with isolated Schoen--Wolfson singularities \cite{GaiaOrriolsRiviere2024}; Gaia constructed branched immersions with infinitely many such singularities accumulating at the boundary \cite{Gaia2025}; and Orriols proved an Allard-type theorem for Legendrian area-minimizing currents \cite{Orriols2026}. Here, however, we restrict throughout to (single-valued) multiplicity-one Lipschitz graphs.

\subsection{Overview of the proofs}\label{sec:proof-overview}

\subsubsection{Regularity results}
We first use the uniform Hessian bound in the weak Hamiltonian stationary equation to obtain a uniformly elliptic PDE in divergence form, and De Giorgi--Nash--Moser theory gives $\beta\in C^{0,\alpha}$. Since $\beta=\FSL(x,Du,D^2u)$, the potential is then a viscosity solution of a uniformly elliptic second-order equation with H\"older right-hand side. Via version of Savin's ``small perturbation" $\eps$-regularity theorem \cite{Savin2007} relative to a quadratic polynomial, in the inhomogeneous form proved by Fan \cite{Fan2026}, we show regularity of $u$ whenever it is sufficiently close to a quadratic polynomial. We may thus characterize the regular set by this closeness property; its complement is the relatively closed set $\Sigma_u$.

We next study the behavior of $L_u$ around a point $x_0 \in \Sigma_u$. The H\"older continuity of the angle makes $\Rea(e^{-i\beta(x_0)}\Omega)$ close to a calibration, and then a comparison argument gives the almost-monotone quantity
\[
  r\longmapsto e^{Cr^\delta}r^{-n}\Mass(\llbracket L_u \mathbin{\llcorner} B_r(x_0, Du(x_0))\rrbracket)
\]
for a suitable choice of $C>0$ and $\delta >0$. It follows that every blow-up of $\llbracket L_u \rrbracket$ is a Lipschitz graphical special Lagrangian cone. Federer's dimension-reduction argument can then be iterated within this class of cones, giving $\dim_{\cH}\Sigma_u\le n-5$ when combined with the known classification of special Lagrangian cones in low dimensions.

\subsubsection{Singular graphical special Lagrangian cone} 
The construction in \cref{thm:counterexample-main} is inspired by the examples of Nadirashvili--Tkachev--Vl\u{a}du\c{t} \cite{NadirashviliTkachevVladut2012}; see also \cite{NadVla11-1,NadVla11-2}. We use the $\mathrm{SO}(3)$-invariant Cartan cubic $P$ on $\R^5$ and look for a two-homogeneous potential of the form
\begin{equation}\label{e:Cartan-cubic}
  u(x)=|x|^2 f\!\left(\frac{P(x)}{|x|^3}\right).
\end{equation}
There is an important difference from the earlier constructions. There, the equation could be adapted to the singular solution, while here the special Lagrangian equation is fixed in advance. The ansatz \eqref{e:Cartan-cubic} reduces the equation to a singular second-order ODE for $f$ on $[-1,1]$. Regularity at the focal orbit $s=1$ gives a one-parameter family of local real-analytic solutions, parametrized by the value $A=f(1)$. We then look for a value of $A$ for which the corresponding solution satisfies $f(0)=0$.

The main point is to show that these local solutions continue all the way from $s=1$ to $s=0$. We do this by a computer-assisted shooting argument using interval arithmetic. For two nearby values $A_-$ and $A_+$, the resulting endpoint values $f_{A_-}(0)$ and $f_{A_+}(0)$ have opposite signs. Continuous dependence on $A$ then gives a value $\widehat A\in(A_-,A_+)$ with $f_{\widehat A}(0)=0$. The symmetry of the reduced equation allows us to reflect this solution across $s=0$, giving a real-analytic profile on the whole interval $[-1,1]$.

\begin{remark}\label{rem:phase-zero-and-convexity}
The odd-reflection step described above is specific to the phase-zero special Lagrangian equation. The resulting profile is odd and $P(-x)=-P(x)$, so $D^2u(-x)=-D^2u(x).$ Since $\FSL(-M)=-\FSL(M)$, a constant-phase equation $\FSL(D^2u)=\theta_0$ can be preserved by this reflection only if $\theta_0=0$ and thus this approach does not extend to nonzero phases. It is unclear whether a shooting argument carried out directly on the full interval $[-1,1]$, without using reflection, could succeed for other subcritical phases.
\end{remark}

\subsection*{Organization}
The paper is organized as follows. In \cref{sec:notation}, we record notation and preliminary facts. In \cref{sec:lag-angle}, we derive the weak generalized Lagrangian angle equation in an almost K\"ahler manifold and prove H\"older continuity of the angle. In \cref{sec:eps-reg}, we use the $\eps$-regularity result of Savin \cite{Savin2007} and Fan \cite{Fan2026}, re-centered around a homogeneous quadratic polynomial, to define the singular set. In \cref{sec:monotonicity}, we derive almost-monotonicity of mass ratios for Lagrangian graphs with H\"{o}lder angle and complete the proof of \cref{thm:regularity-main} via dimension reduction. In \cref{sec:counterexample}, we construct the five-dimensional singular cone and its higher-dimensional products. 

\subsection*{Acknowledgements}
The authors thank Camillo De Lellis, Alessio Figalli, and Yu Yuan for helpful comments. A.B. acknowledges the support of NSF grant DMS-2350290, the Simons Foundation grant MPS-TSM-00002933, and a Bill Guthridge Fellowship from UNC-Chapel Hill. A.S.~is grateful for the generous support of Dr.~Max R\"ossler, the Walter Haefner Foundation, and the ETH Z\"urich Foundation. This research was conducted during the period A.S.~served as a Clay Research Fellow. {Part of this work was conducted while G.O.~and A.S.~were visiting A.B. at UNC-Chapel Hill. They thank the UNC Mathematics Department for their hospitality.}

\subsection*{Note on LLM Usage}
AI tools (ChatGPT-5.6 Sol and Codex) were used to assist with the computer-assisted proof presented in Subsection \ref{global_sol} and with some tedious but routine computations in \cref{sec:counterexample}. The accompanying code was developed through a long sequence of interactions between the authors and these AI tools. The computations referred to above were reworked and rewritten by the authors, who independently checked and verified all associated code and outputs. All the text in the article was written by the authors.\\

\section{Notation and conventions}\label{sec:notation}

We write $B_r(x)$ for the Euclidean ball of radius $r$ centered at $x \in \R^n$,
$B_r=B_r(0)$, and $\omega_n=|B_1|$. We will write $B_r^n(x)$ and $B_r^{2n}(x)$ respectively to distinguish between such balls in $\R^n$ and $\R^{2n}$ when it is necessary. The space of real
symmetric $n\times n$ matrices is denoted by $\Sym(n)$. For
$A,B\in\Sym(n)$, the notation $A\le B$ means that $B-A$ is positive
semidefinite. Constants denoted by $C$ may change from line to line and depend only on the quantities explicitly stated. We use $\langle \cdot, \cdot \rangle$ to denote the inner product of two vectors, induced by a metric that will be clear from context. We will use the notation $\Bbf_R$ for a ball of radius $R$ centered at the origin in a normed space.

Recall that the graph of a Lipschitz map $G : \R^n \to \R^m$, denoted $\graph (G)$, has a canonical orientation and thus induces a multiplicity-one integral current $\llbracket \graph(G) \rrbracket$. We will often work with conical currents of this type, which are of the form $\llbracket \graph (Dw) \rrbracket$ for a $C^{1,1}_{\mathrm{loc}}$ potential function $w$ on $\R^n$. We will use the terminology ``graphical special Lagrangian cone with $C^{1,1}$ potential" for such cones.

We will refer to $\C^n$ as the standard complex vector space $(\R^{2n}, J_0)$ endowed with its standard metric $g_0$ and symplectic form $\omega_0$. When we work locally in a symplectic manifold $(M, \omega)$, we will use a Darboux chart $\Psi : B_r \subset (\R^{2n}, \omega_0) \to M$ and identify tensors in $M$ with their pullback to $\R^{2n}$.

We use $\Mass(T)$ for the mass of a current $T$ in a Riemannian manifold $(M, g)$, $\spt T$ for its support, and $\Sing T$ for the (interior) singular set, defined as the complement of the open set where its support is smooth with locally constant multiplicity; the latter is the (interior) regular set $\Reg T$. For a relatively open set $U \subset M$, $T \llcorner U$ denotes the current given by the restriction of $T$ to $U$. For a.e.~$r>0$, we use the notation $\langle T,|\cdot|,r\rangle$ for the current given by the radius-$r$ spherical slice of $T$. We refer the reader to \cite{SimonGMT} for more details concerning the relevant GMT background and notation.

For closed sets $E_j,E\subset\R^n$, local Hausdorff convergence means
Hausdorff convergence of $E_j\cap\overline B_R$ to
$E\cap\overline B_R$ for every $R$. We write $\cH^s_\infty$ for the
$s$-dimensional Hausdorff content.

An almost K\"ahler manifold $(M^{2n}, \omega, J, g)$ is a symplectic manifold $(M, \omega)$ together with an almost complex structure $J$ such that $g := \omega(\cdot, J \cdot)$ defines a Riemannian metric. If $J$ is integrable, we speak of a K\"ahler manifold, and if it further admits a normalized holomorphic $(n, 0)$ form $\Omega$ it is Calabi--Yau.

\section{The Lagrangian angle beyond Calabi--Yau manifolds}\label{sec:lag-angle}
Our starting point for the regularity of Hamiltonian-stationary submanifolds is a generalized notion of Lagrangian angle that can be defined in arbitrary almost K\"ahler manifolds, and satisfies a second-order elliptic PDE in divergence form.

\subsection{Geometric setting}
Let $(M^{2n}, \omega)$ be a symplectic manifold. Recall that $M$ can be covered by Darboux charts, that is, by the images of  symplectomorphisms $\Psi : \mathcal{U} \subset (\R^{2n}, \omega_0) \xrightarrow{\sim} U \subset (M, \omega)$, where $\omega_0$ is the standard symplectic structure on $\R^{2n}$.

For us, a Lagrangian Lipschitz submanifold is a set $L \subset M$ which can be locally written in Darboux charts as a Lipschitz graph \emph{along} a Lagrangian plane, and whose tangent planes are Lagrangian almost everywhere. More precisely, we require that for every $q \in L$ there exists a Darboux chart as above with $q \in L \cap U$, a Lagrangian plane $\Lambda \subset (\R^{2n}, \omega_0)$, a complementary $n$-plane $\Lambda_0 \subset \R^{2n}$ and a Lipschitz map $F : \Lambda_0 \to \Lambda$ such that $L \cap U = \Psi( \{ x + F(x) : x \in \Lambda_0 \} \cap \mathcal{U})$,
and such that the symplectic form $\omega$ vanishes at almost every tangent plane.

In the presence of a compatible complex structure on $(\R^{2n}, \omega_0)$ we have a more familiar and convenient description:
\begin{lemma}
    \label{lem:lipschitz-gradient-graph}
    Let $\Lambda \subset (\R^{2n}, \omega_0)$ be a Lagrangian plane and $J_0$ a complex structure on $\R^{2n}$ compatible with $\omega_0$. Set $g_0 := \omega_0(\cdot, J_0 \cdot)$ and $\Lambda_0 := J_0 \Lambda$. Then a Lagrangian graph along $\Lambda$ can be written as
    \begin{equation}
        \label{eq:lagrangian-lipschitz-graph}
        \{ x + J_0 \nabla_{g_0} u(x) : x \in \Lambda_0 \}
    \end{equation}
    for some $u \in C^{1,1}(\Lambda_0)$, and conversely any set of the form \eqref{eq:lagrangian-lipschitz-graph} is a Lagrangian Lipschitz graph. Moreover, if $J_1, J_2$ are compatible complex structures, $g_1$ and $g_2$ are the respective metrics, and $\Lambda_1 := J_1 \Lambda$, $\Lambda_2 := J_2 \Lambda$ the respective complementary planes, then the respective potentials $u_1$ and $u_2$ differ (up to a constant) by a quadratic function.
    \begin{proof}
        The first part is standard: we have for all vectors $x, y \in \Lambda_0$ a.e.
        \begin{align*}
            0
            &= \omega_0(x + D F \cdot x, y + D F \cdot y) \\
            &= \omega_0(x, y) + \omega_0( D F \cdot x, D F \cdot y) + \omega_0(x, D F \cdot y) + \omega_0(D F \cdot x, y) \\
            &= -g_0(x, D (J_0 F) \cdot y) + g_0(D(J_0 F) \cdot x, y),
        \end{align*}
        thus $J_0 F : \Lambda_0 \to \Lambda_0$ is the gradient of a $C^{1,1}$ function $-u$ with respect to $g_0$. The converse follows similarly.

        For the last part, $\Lambda_1$ and $\Lambda_2$ are both complements of $\Lambda$, hence $\Lambda_2$ can be written as a graph over $\Lambda_1$ along $\Lambda$, that is, $\Lambda_2 = \{ x + S x : x \in \Lambda_1 \}$ for some linear map $S : \Lambda_1 \to \Lambda$. Now by the first part, $S x = J_1 \nabla_{g_1} w$ for some $w \in C^{1,1}(\Lambda_1)$ and it follows that $w$ is quadratic, up to a constant. Finally, given a potential $u_2 \in C^{1,1}(\Lambda_2)$, if we let $u_1(x) := u_2(x) + w(x)$ for $x \in \Lambda_1$, where we extend $u_2$ from $\Lambda_2$ to the whole $\R^{2n}$ by making it constant on $n$-planes parallel to $\Lambda$, its Hamiltonian vector field is $J_1 \nabla_{g_1} u_2(x) = J_2 \nabla_{g_2} u_2(x + Sx)$ and hence
        \begin{align*}
            \{ x + J_1 \nabla_{g_1} u_1(x) : x \in \Lambda_1 \}
            &= \{ x + J_2 \nabla_{g_2} u_2(x + Sx) + S x : x \in \Lambda_1 \} \\
            &= \{ y + J_2 \nabla_{g_2} u_2(y) : y \in \Lambda_2 \},
        \end{align*}
        showing that indeed $u_1$ is the corresponding potential for $\Lambda_1$.
    \end{proof}
\end{lemma}

This shows that a gradient graph $\{ x + J_1 \nabla_{g_1} u(x) : x \in \Lambda_1 \}$ with respect to one choice of compatible complex structure $J_1$ and Lagrangian plane $\Lambda_1$ in $(\R^{2n}, \omega_0)$ remains so with respect to another choice of complex structure $J_2$, provided that we change the domain of the new graph to $\Lambda_2 := J_2 J_1 \Lambda_1$, and the $L^\infty$ norm of the Hessian of the potential only changes by a constant depending on $J_1$ and $J_2$. This fact (and its local version) will be used often throughout.

\subsection{Lagrangian angle and first variation formula}
On each small enough neighborhood of $M$ we can find a smooth $(n, 0)$-form $\Omega$ of norm $2^{n/2}$, or equivalently, of comass one; we will call such $\Omega$ a normalized $(n, 0)$-form. Notice that we are not imposing any holomorphicity or parallelism condition, so $e^{i \varphi} \Omega$ is also of this class for any smooth real-valued function $\varphi$.

For any oriented Lagrangian plane $\Lambda \subset T_p M$, we may use $\Omega$ to define a notion of Lagrangian angle of $\Lambda$ via $\Omega|_{\Lambda} = e^{i \beta(p, \Lambda)} \vol_\Lambda$. For a Lagrangian submanifold $L$, this defines a function $e^{i\beta} : L \to \Sbb^1$. If $L$ is a Lipschitz graph, $\beta$ is automatically single-valued, and we will assume throughout that $\beta \in W^{1,2}_\mathrm{loc}(L)$.

Recall that $L$ is Hamiltonian-stationary if for any $h \in C^\infty_c(U)$, the first variation of volume of $L$ along the Hamiltonian vector field $X = J \nabla h$ associated to $h$ vanishes. If $\phi_t$ denotes the flow of $X$, using the definition of $\beta$ we have:
\begin{align*}
\left. \frac{\dd}{\dd t} \right|_{t=0} \operatorname{Area}(\phi_t(L))
&= \left. \frac{\dd}{\dd t} \right|_{t=0} \int_{\phi_t(L)} e^{-i \beta_t} \Omega
= -i \int_{L} \dot \beta e^{-i \beta} \Omega + \int_{L} e^{-i \beta} \mathcal{L}_X \Omega \\
&= -i \int_{L} \dot \beta \vol_L + \int_{L} e^{-i \beta} (X \slot \dd \Omega + \dd (X \slot \Omega))\,,
\end{align*}
where $\Dot{\beta} := \frac{\dd}{\dd t} (\beta_t \circ \phi_t)\Big|_{t=0}$ and $\Lcal_X$ denotes the Lie derivative along $X$. We integrate by parts the last term on the right-hand side above and use the complex-linearity of $\Omega$, after pointwise splitting $X = X^\top + X^\perp$ into its tangent and normal parts to $L$:
\begin{align*}
    \int_{L} e^{-i \beta} \dd (X \slot \Omega)
    &= -\int_{L} (-i e^{-i \beta}) \dd \beta \wedge (X \slot \Omega) \\
    &= \int_{L} e^{-i \beta} (i \, \dd \beta \wedge (X^\top \slot \Omega) + \dd \beta \wedge (J X^\perp \slot \Omega)) \\
    &= \int_{L} e^{-i \beta} (i \, \langle X^\top, \dd \beta \rangle \Omega  + \langle J X^\perp , \dd \beta \rangle \Omega ) \\
    &= \int_{L} (i \, \langle X^\top, \dd \beta \rangle + \langle J X^\perp , \dd \beta \rangle ) \vol_L,
\end{align*}
using the fact that, for a vector $Y$ tangent to $L$,
\[
  \left( \dd \beta \wedge (Y \slot \Omega) \right) |_{L}
   =  \langle Y, \dd \beta \rangle \Omega|_L - \left( Y \slot (\dd \beta \wedge \Omega) \right)|_L
   =  \langle Y, \dd \beta \rangle \Omega|_L.
\]
For the same reason, only $e^{-i \beta} X^\perp \slot \dd \Omega = e^{-i \beta} (J \nabla^L h) \slot \dd \Omega$ survives in the $e^{-i \beta} X \slot \dd \Omega$ term above. Hence, taking the real part in the first variation formula we get
\begin{align*}
\left. \frac{\dd}{\dd t} \right|_{t=0} \operatorname{Area}(\phi_t(L))
&= \int_{L} \Rea \left[ e^{-i \beta} (J \nabla^L h) \slot \dd \Omega \right] - \langle \nabla^L h , \dd \beta \rangle \vol_L
= \int_{L} \langle \nabla^L h, \tau - \dd \beta \rangle
\end{align*}
where $\tau$ is a real $1$-form of class $L^\infty$ on $L$ defined via $\tau(Y) \, \vol_L := \Rea[e^{-i \beta} (J Y) \slot \dd \Omega]|_L$. Comparing with the ordinary first variation formula yields $ H \slot \omega=\tau-\dd\beta$ for the mean curvature $H$, and Hamiltonian stationarity is therefore equivalent to
\begin{equation}\label{eq:beta-divergence-geometric}
  \dd_L^*(\dd\beta-\tau)=0
\end{equation}
in the weak sense.
This formulation is independent of the unit trivialization: if $\Omega'=e^{i\varphi}\Omega$, then
$\beta'=\beta+\varphi|_L$ and $\tau'=\tau+\dd_L\varphi$.
In the K\"ahler case one may write $\dd\Omega=\gamma\wedge\Omega$ with $\gamma$ of type $(0,1)$, and a direct calculation gives $\tau = 2\Ima\gamma|_L$.  For a parallel Calabi--Yau volume form the $1$-form $\tau$ vanishes, recovering the classical formula $\Delta_L\beta=0$.

\subsection{H\"older regularity of the angle}
Next we fix a Darboux chart and write $L = L_u = \{ (x, D u(x)) : x \in B_2 \}$ in these coordinates for a function $u \in C^{1,1}(B_2, \R)$. Writing $\beta \equiv \beta \circ \Phi_u$ and $g_u=\Phi_u^*g$, the Lagrangian angle is given by a fully nonlinear operator $\beta(x) = \FSL(x, D u(x), D^2 u(x))$; see Subsection \ref{ss:viscosity-setup} below for more details. Moreover, the Hamiltonian stationarity condition \eqref{eq:beta-divergence-geometric} becomes
\begin{equation}\label{eq:coordinate-beta}
  \int_{B_2} A^{ij}(x)\bigl(\partial_i \beta - b_i \bigr) \partial_j h \,\dd x=0 \quad \forall h \in C^\infty_c(B_2),
  \qquad
  A^{ij}=\sqrt{\det g_u}\,g_u^{ij},
\end{equation}
where $b$ is the pullback of the bounded one-form $\tau$ via $\Phi_u$.  A bound for $D^2u$ and the smooth bounded geometry of the chart imply
\begin{equation*}\label{eq:A-elliptic}
  \lambda |\xi|^2\le A^{ij}(x)\xi_i\xi_j\le\Lambda|\xi|^2,
  \qquad \|b\|_{L^\infty}\le C.
\end{equation*}

\begin{lemma}\label{lem:angle-holder}
In the above setting, there are $\alpha\in(0,1)$ and $C<\infty$, depending only on $n$, the local Hessian bound, and the ambient geometry, such that
\begin{equation*}\label{eq:angle-holder-estimate}
  [\beta]_{C^{0,\alpha}(B_{3/2})}
  \le C\bigl(1+\|\beta\|_{L^2(B_2)}\bigr).
\end{equation*}
\end{lemma}

\begin{proof}
Equation \eqref{eq:coordinate-beta} is a uniformly elliptic divergence-form equation with right-hand side the divergence of the bounded field $A b$.  The De Giorgi--Nash--Moser interior estimate with $L^q$ divergence data, for any $q>n$, gives the above estimate; see, for example, \cite[Theorem~8.24]{GilbargTrudinger}.
\end{proof}

\subsection{The second-order angle operator}\label{ss:viscosity-setup}
Next we study the fully nonlinear operator that gives the generalized Lagrangian angle of a gradient graph in Darboux coordinates on an arbitrary almost K\"ahler manifold.

Let $\Psi : B_r \subset \C^n \to M$ be a Darboux chart around a point $q = \Psi(0, 0) \in M$ in which a normalized $(n, 0)$-form $\Omega$ is defined, and consider the pullback tensors $g_0 := \Psi^* g(0)$, $J_0 := \Psi^* J(0)$ and $\Omega_0 := \Psi^*\Omega(0)$. After multiplying by a global phase, we can assume that $J_0, g_0, \Omega_0$ and $\omega_0 := \Psi^* \omega$ are the standard Calabi--Yau structure on $\C^n$; we will omit the pullback by $\Psi$ to simplify the notation, and using the compatibility of $\omega_0$, $J$ and $g$, decompose
\begin{equation}
\label{eq:J-and-g}
    J = J(x, p) =
    \begin{pmatrix}
        J_{11} & J_{12} \\
        J_{21} & J_{22}
    \end{pmatrix}
    \quad \text{and} \quad
    g = g(x, p)
    =
    \begin{pmatrix}
        g_{11} & g_{12} \\
        g_{21} & g_{22}
    \end{pmatrix}
    =
    \begin{pmatrix}
        J_{21} & J_{22} \\
        -J_{11} & -J_{12}
    \end{pmatrix}.
\end{equation}
In this setting, we define\footnote{Note that there a canonical smooth choice of $\beta$ among graphical Lagrangian planes with $\beta(0; \R^n) = 0$.} the operator $\mathcal{F} : B_r \times \Sym(n) \to \R$ by
\begin{equation}\label{eq:variable-phase-equation}
  \FSL(x, p, M) := \beta(x, p; \operatorname{span}\{ (e_1, M e_1), \dots, (e_n, M e_n)\}).
\end{equation}

\begin{proposition}
There exists $r$ sufficiently small such that the operator $\FSL$ is locally uniformly elliptic on $B_r$, namely
\begin{equation}\label{eq:F-ellipticity}
  \lambda_K \tr N \le
  \FSL(x,p, M+N) - \FSL(x,p,M)
  \le\Lambda \tr N
\end{equation}
for constants $0 < \lambda_K < \Lambda < \infty$ whenever $(x,p) \in B_r$ and $\|M\|, \|N\| \le K$.
\begin{proof}
We fix $(x, p) \in B_r$ and compute the derivative of the Lagrangian angle of $\operatorname{span} \{ v_1, \dots, v_n \}$ at $(x, p)$ for an arbitrary $1$-parameter family of linearly independent vectors $\{ v_i(t) \}$: if we write
\begin{equation}
\label{eq:v-dot-A-B}
    \dot v_i = \sum_{j=1}^n a_{ji} v_j + \sum_{j=1}^n b_{ji} J v_j
\end{equation}
and $\Omega(v_1(t), \dots, v_n(t)) = e^{i \beta(t)} w(t)$ with $w(t) > 0$, then
\begin{align*}
     \frac{\dot w}{w} + i \dot \beta
    &= (e^{i \beta} w)^{-1} \left.\frac{\dd}{\dd t}\right|_{t=0} \Omega( v_1(t), \dots, v_n(t)) \\
    &= (e^{i \beta} w)^{-1} \left[ \Omega( \dot v_1, \dots, v_n) + \dots + \Omega( v_1, \dots, \dot v_n) \right] \\
    &= (e^{i \beta} w)^{-1} \left( 
    (a_{11} + \dots + a_{nn}) \Omega( v_1, \dots, v_n) 
    + i (b_{11} + \dots + b_{nn}) \Omega( v_1, \dots, v_n) 
    \right) \\
    &= \tr A + i \tr B,
\end{align*}
thanks to the $\C$-linearlity of $\Omega$, so taking imaginary parts we get $\dot \beta = \tr B$. We apply this to $v_i(t) = (e_i, M(t) e_i)$ for $M(t)$ a path in $\Sym(n)$. In this case, the matrices $A$ and $B$ in \eqref{eq:v-dot-A-B} solve the system
\[
\begin{pmatrix} 0 \\ \dot M \end{pmatrix}
= \begin{pmatrix} \Id \\ M \end{pmatrix} A + \begin{pmatrix} J_{11} & J_{12} \\ J_{21} & J_{22} \end{pmatrix} \begin{pmatrix} \Id \\ M \end{pmatrix} B,
\]
so that $A = - (J_{11} + J_{12} M) B$ and
\begin{equation}
\label{eq:B-from-Mdot}
    B = [J_{21} + J_{22} M - M J_{11} - M J_{12} M ]^{-1} \dot M.
\end{equation}
then \eqref{eq:B-from-Mdot} and \eqref{eq:J-and-g} show that
\[
    D_M \mathcal{F}(x, p, M)[N] = \tr \left( [g_{11} + g_{12} M + M g_{21} + M g_{22} M ]^{-1} N \right).
\]
The eigenvalues of $g_{11} + g_{12} M + M g_{21} + M g_{22} M$ are clearly bounded from above when $\| M \| \leq K$; to see that they are bounded from below (uniformly in $M$), observe that after shrinking $r$ if necessary we may assume that $\| g - g_0 \| \leq \tfrac{1}{4}$, so that for any $v \in \R^n$,
\begin{align*}
    &g_{11}(v, v) + g_{12}(v, Mv) + g_{21}(Mv, v) + g_{22}(M v, M v) \\
    &\quad \geq \frac{3}{4} |v|^2 - 2 \cdot \frac{1}{4} |v| |M v| + \frac{3}{4} |M v|^2
    \geq \frac{1}{2} |v|^2 + \frac{1}{2} |M v|^2
    \geq \frac{1}{2} |v|^2.
\end{align*}
Thus, integrating $D_M\FSL(x,p,M)[N]$ along $M+tN$ for $N \geq 0$ yields \eqref{eq:F-ellipticity}.
\end{proof}
\end{proposition}

\section{\texorpdfstring{$\eps$-regularity and the singular set}{epsilon-regularity and the singular set}}\label{sec:eps-reg}

Here, we use Fan's inhomogeneous version \cite[Theorem~1.7]{Fan2026} of Savin's ``small perturbation'' $\eps$-regularity result \cite{Savin2007} (see also \cite{BS_optimal_reg}), in the case of closeness to a homogeneous quadratic polynomial. We begin by recalling some basic setup for general viscosity solutions of second-order fully nonlinear equations.

\subsection{Viscosity solutions}
We use the convention of \cite{Fan2026} for an operator
$G(x,p,A)$ that is nondecreasing in $A$ (with respect to the partial ordering of symmetric matrices). A viscosity subsolution of
$G=0$ satisfies
\[
  G(x_0,D\varphi(x_0),D^2\varphi(x_0))\ge0
\]
whenever $\varphi\in C^2$ touches it from above, while a viscosity
supersolution satisfies the reverse inequality whenever $\varphi$
touches it from below. Next, we show that almost-everywhere solutions are viscosity solutions.

\begin{lemma}\label{lem:ae-viscosity}
Let $G:\R^n\times\R^n\times\Sym(n)\to\R$ be continuous and
nondecreasing with respect to its matrix variable. If $v\in C^{1,1}_{\mathrm{loc}}$
satisfies
\begin{equation}\label{e:ae-eq}
  G(x,Dv,D^2v)=0
\end{equation}
almost everywhere, then $v$ is a viscosity solution of $G=0$.
\end{lemma}

\begin{proof}
Let $E$ be the full Lebesgue measure set on which $v$ is twice differentiable and
the equation \eqref{e:ae-eq} holds. Suppose $\varphi\in C^2$ touches $v$ from above at
$x_0$ and that the required subsolution inequality fails. After replacing
$\varphi$ by
$\psi=\varphi+\frac\varepsilon2|x-x_0|^2$, continuity gives
\begin{equation}\label{e:subsol-contradiction}
  G(x,D\psi(x)+q,D^2\psi(x))<0
\end{equation}
for $x$ in a sufficiently small neighborhood of $x_0$ and $q \in \R^n$ with $|q|$ small, while $v-\psi$ has a strict local
maximum at $x_0$. Jensen's maximum lemma
\cite[Lemma~A.3]{CIL1992} provides a positive Lebesgue-measure set $F$ of nearby
local maxima for small affine perturbations of $v-\psi$. Choosing a point
$y\in E\cap F$, for some small $p$ we therefore have
\[
  Dv(y)=D\psi(y)-p,
  \qquad
  D^2v(y)\le D^2\psi(y).
\]
Combining this with the fact that $y\in E$, the nondecreasing property of $G$ in the third variable, and \eqref{e:subsol-contradiction}, we arrive at
\[
  0=G(y,Dv(y),D^2v(y))
  \le G(y,D\psi(y)-p,D^2\psi(y))<0,
\]
which is a contradiction. The supersolution inequality follows identically by instead using $\psi=\varphi-\frac\varepsilon2|x-x_0|^2$ and applying Jensen's maximum lemma to $\psi-v$.
\end{proof}

Applied to \eqref{eq:variable-phase-equation}, \cref{lem:ae-viscosity} permits the use of fully nonlinear viscosity theory.

For a homogeneous quadratic polynomial $P$, define the corresponding operator and right-hand side that are both re-centered around $P$:
\begin{align}
  \mathcal G_P(x,p,M)
  &:={\FSL}\bigl(x,DP(x)+p,D^2P+M\bigr)
    -{\FSL}\bigl(x,DP(x),D^2P\bigr)\,,
  \label{eq:centered-operator}\\
  h_P(x)
  &:=f(x)-{\FSL}\bigl(x,DP(x),D^2P\bigr)\,.
  \label{eq:centered-rhs}
\end{align}
For $\rho>0$ and $P$ as above, recall the H\"{o}lder seminorm
\begin{equation*}\label{eq:centered-oscillation}
  [\mathcal G_P]_{x;\alpha,\rho}
  :=\sup_{\substack{\|M\|,|p|\le\rho\\x\ne y\in B_1}}
  \frac{|\mathcal G_P(x,p,M)-\mathcal G_P(y,p,M)|}{|x-y|^\alpha}\,.
\end{equation*}

\begin{proposition}\label{prop:epsilon-regularity}
Fix $0<K<\infty$, $\alpha\in(0,1)$, and a $C^2$-compact family $\Oscr_{K+1}$ of operators $\FSL$ that satisfy \eqref{eq:F-ellipticity} with uniform constants on
\[
  \mathcal K_{K+1}:=
  \bigl\{(x,p,A):x\in\overline B_1,\ |p|\le K+1,\ \|A\|\le K+1\bigr\}\,.
\]
There are constants $\eps_0>0$, $C<\infty$ depending only on
$n,K,\alpha$ and the compact family, with the following property.
Suppose
\[
  u\in C^{1,1}(B_1),
  \qquad
  \|Du\|_{L^\infty(B_1)}+\|D^2u\|_{L^\infty(B_1)}\le K,
\]
and suppose that for some $\Fcal \in \Oscr_{K+1}$, $u$ satisfies
\begin{equation}\label{e:viscosity-sol-u}
  \FSL(x,Du,D^2u)=f(x)
\end{equation}
in the viscosity sense. Let
\[
  P(x)=\frac12x^TA_0x,
  \qquad \|A_0\|\le K\,.
\]
Then $w:=u-P$ is a viscosity solution of
\begin{equation}\label{e:quadr-recentering}
  \mathcal G_P(x,Dw,D^2w)=h_P(x)
  \qquad\text{in }B_1,
\end{equation}
where $\mathcal G_P$ and $h_P$ are defined in
\eqref{eq:centered-operator}-\eqref{eq:centered-rhs}. If
\begin{equation}\label{eq:eps-small}
  \|u-P\|_{L^\infty(B_1)}
  +\|h_P\|_{C^{0,\alpha}(B_1)}
  +[\mathcal G_P]_{x;\alpha,1/2}
  \le\eps_0,
\end{equation}
then $u\in C^{2,\alpha}(B_{1/2})$ and
\[
  \|u-P\|_{C^{2,\alpha}(B_{1/2})}\le C.
\]
\end{proposition}

\begin{proof}
The viscosity identity \eqref{e:quadr-recentering} follows directly by
testing the viscosity subsolution (respectively supersolution) inequality for $u$ with $P+ \varphi$, where $\varphi$ is an arbitrary $C^2$ function that touches $w$ from above (respectively below) at an arbitrary point. Let us now verify the remaining
hypotheses of \cite[Theorem~1.7]{Fan2026}. First, we clearly have
\[
  \mathcal G_P(x,0,0)=0.
\]
Moreover, for $M,N\in\Sym(n)$ with $N\ge0$, the local uniform ellipticity of $\FSL$ in turn implies
\[
  \lambda\tr N\le
  \mathcal G_P(x,p,M+N)-\mathcal G_P(x,p,M)
  \le\Lambda\tr N
\]
with constants uniform over $\Oscr_{K+1}$. Observe that all
arguments occurring for $\|M\|,\|N\|,|p|\le 1/2$ and $\|A_0\| \leq K$ remain in $\mathcal K_{K+1}$. Compactness in $C^2$ then gives a uniform
constant $L$ and a common modulus of continuity $\omega$ such that
\begin{align*}
 |\mathcal G_P(x,p,M)-\mathcal G_P(x,q,M)|
 &\le L|p-q|,\\
 \|D_M\mathcal G_P(x,p,M)-D_M\mathcal G_P(y,q,N)\|
 &\le\omega\bigl(\|M-N\|+|p-q|+|x-y|\bigr)
\end{align*}
whenever the displayed variables lie in $\Kcal_{K+1}$. Thus Fan's structure condition holds with $b_0=L$ and $c_0=0$,
since the operator $\Gcal_P$ has no zeroth-order dependence. The final term $[\mathcal G_P]_{x;\alpha,1/2}$ in
\eqref{eq:eps-small} is exactly the small spatial oscillation of the
centered operator required in Fan's theorem, while the first two terms give the required smallness of the solution $w$ and of the right-hand side.
All structural constants are uniform over $\Oscr_{K+1}$ and over
$\|A_0\|\le K$. Fan's theorem therefore applies to $w$ and yields the claimed $C^{2,\alpha}$ estimate.
\end{proof}

\subsection{The singular set}
Let us now use the $\eps$-regularity result in Proposition \ref{prop:epsilon-regularity} to define the singular set $\Sigma_u$ in \cref{thm:regularity-main}. We will then verify the claimed dimension estimate on it in \cref{sec:monotonicity}, after first obtaining almost-monotonicity for mass ratios in order to obtain conical blowups at singularities. Let $x_0 \in \Ucal$, let $p_0:=Du(x_0)$ and define the rescalings
\begin{equation}\label{eq:rescaled-angle-operator}
  \FSL_{x_0,r}(y,p,A)
  :=\FSL(x_0+ry,p_0+rp,A)\,,
  \qquad \beta_{x_0,r}(y):=\beta(x_0+ry)\,.
\end{equation}
If a given homogeneous quadratic polynomial $P$ satisfies
\begin{equation}\label{eq:rescaled-frozen-normalization}
  \FSL(x_0,p_0,D^2P)=\beta(x_0),
\end{equation}
then smoothness of $\FSL$ and \cref{lem:angle-holder} give
\begin{equation*}\label{eq:rescaled-centered-smallness}
  [\mathcal G_{P,r}]_{y;\alpha,1/2}\le Cr,
  \qquad
  \|h_{P,r}\|_{C^{0,\alpha}(B_1)}
  \le C\bigl(r+r^\alpha[\beta]_{C^{0,\alpha}}\bigr)\,,
\end{equation*}
where $\mathcal G_{P,r}$ and $h_{P,r}$ are formed from $\FSL_{x_0,r}$ and $f = \beta_{x_0,r}$ as in \eqref{eq:centered-operator}-\eqref{eq:centered-rhs}.  Thus \cref{prop:epsilon-regularity} applies at every sufficiently small scale once the rescaled potential
\begin{equation}\label{eq:potential-blowup}
  u_{x_0,r}(y) :=\frac{u(x_0+ry)-u(x_0)-rDu(x_0)\cdot y}{r^2}.
\end{equation}
is sufficiently close in $L^\infty$ to such a $P$.

Now, let $\cQ$ denote the set of homogeneous quadratic polynomials, and define the {singular set} of $u$ as
\begin{equation}\label{eq:singular-set-definition}
  \Sigma_u:=\left\{x_0:\liminf_{r\downarrow0}\inf_{P\in\cQ}
  \|u_{x_0,r}-P\|_{L^\infty(B_1)}>0\right\}\,.
\end{equation}

If the liminf in \eqref{eq:singular-set-definition} vanishes, one may choose a sequence of scales $r_j\downarrow0$ and homogeneous quadratics $P_j$ such that
\[
  \|u_{x_0,r_j}-P_j\|_{L^\infty(B_1)}\longrightarrow0\,.
\]
The normalization $u_{x_0,r_j}(0)=0$, $Du_{x_0,r_j}(0)=0$ and the uniform Hessian bound on $u_{x_0,r_j}$ over $B_1$ imply that the coefficients of $P_j$ stay bounded.  Passing to a subsequence, $P_j\to P_0$ and $u_{x_0,r_j}\to P_0$ uniformly on $B_1$.  Moreover, passing to the limit in the rescaled equations \eqref{eq:rescaled-angle-operator} yields
\[
  \FSL(x_0,p_0,D^2P_0)=\beta(x_0).
\]
Thus, $P_0$ satisfies \eqref{eq:rescaled-frozen-normalization} and so the above discussion guarantees that \cref{prop:epsilon-regularity} applies to $u_{x_0,r_j}$ with the quadratic $P_0$ for all sufficiently large $j$.

\cref{prop:epsilon-regularity} then implies that the complementary set $\Ucal\setminus\Sigma_u$, which we will refer to as the regular set of $u$, is open and $u\in C^{2,\alpha}$ there.  Standard Schauder bootstrapping gives smoothness when the ambient space is $\C^n$.  In an ambient smooth almost K\"{a}hler manifold, the graph is then $C^{1,\alpha}$ and \cite[Theorem 1.1]{BCW2023} gives smoothness. Alexandrov's Theorem shows that $\Sigma_u$ has Lebesgue measure zero.

We conclude this section with the following persistence of singularities result, which is a simple consequence of Proposition \ref{prop:epsilon-regularity} and will come in useful later. We use the notation $C^{0,\alpha}_x$ below to denote H\"{o}lder regularity with respect to the $x$-variable.

\begin{lemma}\label{lem:singular-persistence}
Let $0 < K < \infty$ and $\alpha \in (0,1)$. Let $\{u_j\} \subset C^{1,1}(B_1)$ be a sequence of viscosity solutions to equations of the form \eqref{e:viscosity-sol-u} with $C^{0,\alpha}_x$ operators $\{\Fcal_j\}\subset \Oscr_{K+1}$ as defined in \cref{prop:epsilon-regularity}, and with $C^{0,\alpha}$ right-hand sides $f_j$. Suppose in addition that $\sup_j \|u_j\|_{C^{1,1}(B_1)} \leq K$. Assume that
$u_j\to u_\infty$ in $C^1_{\mathrm{loc}}$, and that for $\alpha \in (0,\alpha_0)$ the quantities $\Fcal_j(\cdot, Du_j(\cdot),D^2 u_j(\cdot))$, $f_j$ converge in $C^{0,\alpha}_x$ to a limiting equation and right-hand side, respectively. If
$x_j\in\Sigma_{u_j}$ and $x_j\to x_\infty$, then
$x_\infty\in\Sigma_{u_\infty}$.
\end{lemma}

\begin{proof}
Suppose for a contradiction that $x_\infty\notin\Sigma_{u_\infty}$. Then $x_\infty$ lies in the
regular set $B_1 \setminus \Sigma_{u_\infty}$ of $u_\infty$, so by \cref{prop:epsilon-regularity}, there exists a sufficiently small radius $r>0$
and a homogeneous quadratic polynomial $P$ such that, after rescaling at
$x_\infty$ by $r$, the three quantities in \eqref{eq:eps-small} are each
strictly below $\eps_0/2$ on $B_1$. 

For this fixed $r$, the convergence $x_j\to x_\infty$ and
$u_j\to u_\infty$ in $C^1$ implies
\[
  (u_j)_{x_j,r}\longrightarrow (u_\infty)_{x_\infty,r}
  \qquad\text{uniformly on }B_1.
\]
The rescaled operators and right-hand sides converge uniformly as well, with the
same ellipticity and $C^2$ bounds. Hence, for all sufficiently large $j$,
the rescaled equation at $x_j$ satisfies \eqref{eq:eps-small} with the
same polynomial $P$ and with threshold $\eps_0$. Proposition
\ref{prop:epsilon-regularity} may then be applied to all such $x_j$, in turn implying that each one is a regular point of $u_j$,
contradicting the fact that $x_j\in\Sigma_{u_j}$. Thus we indeed have
$x_\infty\in\Sigma_{u_\infty}$ as claimed.
\end{proof}

\section{Almost-monotonicity and dimension reduction}\label{sec:monotonicity}
In this section, we demonstrate almost-monotonicity of mass ratios of a {Hamiltonian-stationary} Lagrangian graph in an almost K\"{a}hler manifold, which will in turn allow us to obtain conical blowups.

For the rest of this section we work in a Darboux chart $\Psi : B_r \subset (\R^{2n}, \omega_0) \to (M, \omega)$ where a normalized $(n, 0)$-form $\Omega$ is defined and write $L = L_u$ as in \eqref{eq:DW-graph}. We will omit $\Psi$ and treat $g$, $J$, $\omega$ and $\Omega$ as tensors on $\R^{2n}$; in particular $\omega = \omega_0 = \sum_i \dd x^i \wedge \dd p_i$. Here $\Mass$ denotes the mass with respect to the metric $g$.

\begin{proposition}\label{prop:almost-monotonicity}
Let $T$ be the multiplicity-one integral current induced by an oriented Hamiltonian-stationary Lagrangian Lipschitz submanifold $L$ in a smooth almost K\"ahler manifold $(M^{2n},\omega,J,g)$, and suppose its generalized angle $\beta$ is $W^{1,2}_\mathrm{loc}$.
Then for every interior point $p\in\spt T$ there exist constants $r_0(p,n), C(n), \delta(n) > 0$ such that
\begin{equation*}\label{eq:almost-monotone}
  r\longmapsto e^{C r^\delta}\frac{\Mass(T\llcorner B_r(p))}{r^n}
  \quad\text{is nondecreasing on }(0,r_0).
\end{equation*}
Consequently the density
\begin{equation*}\label{eq:density}
  \Theta^n(T,p):=\lim_{r\downarrow0}\frac{\Mass(T\llcorner B_r(p))}{\omega_n r^n}
\end{equation*}
exists and is finite.
\end{proposition}

\begin{proof}
After multiplying $\Omega$ by a constant phase we may assume that $\beta(p) = 0$. Write
\[
  m(r):=\Mass(T\llcorner B_r)
  \qquad \text{and}
  \qquad S_r:=\langle T,|\cdot|,r\rangle\,;
\]
the latter is well-defined and has finite mass for almost-every $r$ by the slicing theory of currents (see e.g. \cite{SimonGMT}). For almost every $r$, let $R_r = 0 \ttimes S_r$ be the cone over $S_r$. An application of the coarea formula with an error term from the ambient metric yields
\begin{equation}\label{eq:cone-mass}
  \Mass(R_r)\le \frac r n(1+Cr)\Mass(S_r).
\end{equation}
The $n$-dimensional cycle $R_r-T\llcorner B_r$ has an $(n+1)$-dimensional cone filling $W_r$ satisfying
\begin{equation}\label{eq:filling-mass}
  \Mass(W_r)\le Cr\bigl(m(r)+\Mass(R_r)\bigr).
\end{equation}
If $0 < \alpha < 1$ denotes the exponent from \cref{lem:angle-holder} applied in a small enough ball around $p$, we have that $\Rea \Omega(\vec{T}(q)) = \cos(\beta(q)) \geq 1 - C \beta(q)^2 \geq 1 - C|q|^{2\alpha}$ for all $q$ near $p$.
Combining this with \eqref{eq:filling-mass}, we obtain
\begin{align*}
  \Mass(R_r)
  &\ge \langle R_r,\Rea\Omega\rangle\notag\\
  &=\langle T\llcorner B_r,\Rea\Omega\rangle
    +\langle W_r,\dd(\Rea\Omega)\rangle\notag\\
  &\ge (1-Cr^{2\alpha})m(r)
    -Cr\bigl(m(r)+\Mass(R_r)\bigr).
\end{align*}
Rerranging and recalling \eqref{eq:cone-mass}, for $r$ small enough we get
\[
  \frac r n\Mass(S_r)\ge (1-Cr^\delta)m(r)
\]
for $\delta := \min\{1, 2\alpha\}$.
On the other hand, the coarea formula gives $m'(r) \ge (1-Cr) \Mass(S_r)$ for almost every $r$. We thus conclude that
\[
  \frac{\dd}{\dd r}\log\bigl(r^{-n}m(r)\bigr)
  \ge-Cr^{\delta-1}\,,
\]
and the statement follows.
\end{proof}

An important consequence of the almost-monotonicity formula in \cref{prop:almost-monotonicity} is that rescalings of a given current induced by a Hamiltonian stationary Lagrangian graph with $C^{1,1}$-potential converge subsequentially to a cone, referred to as a tangent cone.

\begin{proposition}\label{prop:tangent-cones}
Let $u$ be as in \cref{thm:regularity-main}, let $x_0$ be a point in its domain, and let $T$ be the multiplicity-one current $\llbracket \graph(Du) \rrbracket$.
Then every sequence $r_j\downarrow0$ has a subsequence for which the rescaled potentials \eqref{eq:potential-blowup} converge in
$C^{1,\kappa}_{\mathrm{loc}}(\R^n)$, for every $\kappa<1$, to a function
$v\in C^{1,1}_{\mathrm{loc}}(\R^n)$ satisfying
\begin{equation}\label{eq:frozen-tangent-equation}
  \FSL\bigl(x_0,Du(x_0),D^2v\bigr)=\beta(x_0)
\end{equation}
in the viscosity sense. Setting $q_0 = (x_0, Du(x_0))$, the rescalings $T_j:=(\eta_{q_0,r_j})_\sharp T$ converge weakly without loss of mass to the cone $ C = \llbracket \graph(Dv) \rrbracket $, which is calibrated by $\Rea [e^{-i \beta(q_0)} \Omega(q_0)]$ in $(\R^{2n}, g(q_0), J(q_0))$.

Moreover, letting $(\Sigma_u)_{x_0,r_j} = \eta_{x_0,r_j}(\Sigma_u)$, for any sequence $y_j\in (\Sigma_u)_{x_0,r_j}$ with $y_j \to y$, we have that $y \in \Sigma_v$ and $(y, Dv(y)) \in \Sing C$.
\end{proposition}

In particular, if $x_0 \in \Sigma_u$, by identifying $(\R^{2n}, g(q_0), J(q_0))$ with the standard $\C^n$ and using \cref{lem:lipschitz-gradient-graph}, this procedure yields a singular special Lagrangian tangent cone which is graphical in $\C^n$ in the standard sense.

\begin{proof} 
The uniform Hessian bound for the rescaled potentials gives local compactness of blow-ups, so after passing to a subsequence, $u_{x_0,r_j}\to v$ in $C^{1,\kappa}_{\mathrm{loc}}$ for every $\kappa<1$.
By \cref{lem:angle-holder}, the rescaled angles $\beta_{x_0,r_j}$ converge uniformly to $\beta(x_0)$, while the rescaled operators converge locally to
$A\mapsto\FSL(x_0,Du(x_0),A)$. Classical compactness for viscosity solutions thus yields
\eqref{eq:frozen-tangent-equation}.

Define in addition the rescaled graphing maps
\[
  F_j(y)
  := \left( y, Du_{x_0,r_j}(y)\right),
\]
which are uniformly Lipschitz in $j$ on compact subsets of $\R^n$.
The $C^1_{\mathrm{loc}}$ convergence of the potentials gives
\[
  F_j(y)\longrightarrow F_\infty(y) := (y,Dv(y))
\]
locally uniformly. 

Observe that for any $R>0$, for $j$ sufficiently large we have $T_j\llcorner\Cbf_R = (F_j)_\sharp \llbracket B_R^n\rrbracket$, and $(F_\infty)_\sharp \llbracket B_R^n\rrbracket = C\llcorner \Cbf_R$, where $\Cbf_R$ denotes a cylinder over $B_R^n\subset\R^n \subset \R^{2n}$.
The local uniform convergence together with a linear homotopy argument shows that $T_j \mathbin{\llcorner} B_R^{2n}$ converges in the flat topology, and thus weakly, to $C \mathbin{\llcorner} B_R^{2n}$. We will henceforth consistently work with balls in $\R^{2n}$ and thus omit the superscript.

It is clear that $\Mass(C \llcorner B_R) \leq \liminf_{j \to \infty} \Mass(T_j \llcorner B_R)$.
For the converse inequality, we use the same argument as in the proof of \cref{prop:almost-monotonicity}, based on \cref{lem:angle-holder}. Fix $\epsilon > 0$ and $\chi \in C^\infty(B_{R+\epsilon})$ with $\chi \equiv 1$ on $B_{R}$, and compute
\begin{align*}
    \Mass(C \llcorner B_{R+\epsilon})
    &\geq \langle C, \chi \Rea [e^{-i \beta(q_0)} \Omega] \rangle
    = \lim_{j\to\infty} \langle T_j, \chi \Rea [e^{-i \beta(q_0)} \Omega] \rangle \\
    &\geq \liminf_{j\to\infty} (1 - C (R r_j)^{2\alpha}) \Mass(T_j \llcorner B_{R})
    = \liminf_{j\to\infty} \Mass(T_j \llcorner B_{R})\,.
\end{align*}
The convergence of masses follows after letting $\epsilon \to 0$. A similar argument also establishes that $C$ is calibrated by $\Rea [e^{-i \beta(q_0)} \Omega]$ with respect to the metric $g(q_0)$:
\begin{align*}
    \langle C, \chi \Rea [e^{-i \beta(q_0)} \Omega] \rangle
    &\geq \langle C, \chi \Rea [e^{-i \beta(q_0)} \Omega] \rangle \\
    &= \lim_{j\to\infty} \langle T_j, \chi \Rea [e^{-i \beta(q_0)} \Omega] \rangle \\
    &\geq \liminf_{j\to\infty} (1 - C (R r_j)^{2\alpha}) \Mass(T_j \llcorner B_R)
    \geq \Mass(C \llcorner B_R)\,.
\end{align*}

As a result, $C$ is area-minimizing and
\begin{equation*}\label{eq:rescaled-mass-limit}
  R^{-n} \Mass(C \llcorner B_R)
  = \lim_{j\to\infty} R^{-n}\Mass(T_j\llcorner B_R)
  = \omega_n\Theta^n(T,q_0),
\end{equation*}
which by the equality case in the monotonicity formula for minimal surfaces yields that $C$ is a cone (see e.g. \cite[Chapters 4, 6 \& 7]{SimonGMT} for more details).

Finally, the persistence of singularities is a consequence of \cref{lem:singular-persistence} (the H\"{o}lder convergence in the hypotheses is a direct consequence of Arzel\`{a}-Ascoli): if $y_j\in (\Sigma_u)_{x_0,r_j}$ and $y_j\to y$, the lemma gives directly $y\in\Sigma_v$. The fact that $(y,Dv(y)) \in \Sing C$ follows immediately from Allard's Regularity Theorem and the definition of the latter.
\end{proof}

With the almost-monotonicity of \cref{prop:almost-monotonicity} at hand, we will proceed to perform Federer's dimension reduction argument to conclude the proof of \cref{thm:regularity-main}. To do this, let us first record the following lemma, which tells us that iterated conical blowups remain in our graphical special Lagrangian class.

\begin{lemma}\label{lem:graphical-splitting}
Let $C=\llbracket \graph(Dv)\rrbracket$ be a graphical special Lagrangian cone in $\C^n$ with potential $v \in C^{1,1}(\R^n)$, and let $Y=(y,Dv(y))\in\Sing C\setminus\{0\}$. Any tangent cone $C_Y$ of $C$ at
$Y$ is invariant under translations along $Y$ and splits orthogonally as
\begin{equation}\label{eq:cone-splitting}
  C_Y=\Span \{Y\}\times C_0,
\end{equation}
where, after a unitary identification of
$(\Span\{Y,JY\})^\perp$ with $\C^{n-1}$, the cross-section $C_0$ is a non-flat
multiplicity-one graphical special Lagrangian cone with $C^{1,1}$ potential.
\end{lemma}

\begin{proof}
    Up to a rotation, we may assume that $\spt C$ is a graph over $\R^n \cong \R^n \times \{0\} \subset \C^n$.
For a graph, $Y\ne0$ implies $y\ne0$; we may also normalize $Y$ so that $|Y| = 1$. Given any sequence $r_j\downarrow0$ producing
$C_Y$, consider the normalized potentials
\[
  v_{y,r_j}(z)
  :=\frac{v(y+r_jz)-v(y)-r_jDv(y)\cdot z}{r_j^2}.
\]
The uniform Hessian bound gives, after passing to a subsequence,
$v_{y,r_j}\to v_Y$ in $C^{1,\kappa}_{\mathrm{loc}}$ for every $\kappa<1$,
and the corresponding graph currents converge to
\[
  C_Y=\llbracket \graph(Dv_Y) \rrbracket\,.
\]
Thus $C_Y$ is still a global multiplicity-one current induced by the gradient graph of a potential $v_Y \in C^{1,1}(\R^n)$.
Since $Y\neq 0$, standard cone splitting gives translation invariance along $Y$ and the orthogonal product decomposition \eqref{eq:cone-splitting} with $C_0$ still a Lipschitz graph and a cone. In addition, since $C_Y$ is Lagrangian and $Y$ is tangent to $C_Y$ at every point in its support, the function $\langle J Y, q \rangle$ is constant on $\spt C_Y$ and thus $C_Y$ is supported on $(JY)^\perp$. Hence
\[
  \spt C_0 \subset W := (\Span\{Y,JY\})^\perp
\]
and $C_0$ is still Lagrangian in $W$.
Finally, since the tangent cone $C_Y$ is calibrated by the same form $\Rea \Omega$ as $C$, $C_0$ is calibrated by $\pm (Y \slot \Rea \Omega)$, which is also a special Lagrangian calibration on $W \cong \C^{n-1}$. The fact that $C_Y$ is non-flat as follows from Allard's Regularity Theorem, as $C_Y$ has density one almost everywhere and $Y$ is a singular point of $C$.
\end{proof}

We are now in a position to conclude the proof of \cref{thm:regularity-main}. Given the preservation of graphicality from iterated blowups given by \cref{lem:graphical-splitting}, the remaining argument is a standard dimension reduction argument \`{a} la Federer (see for instance \cite[Appendix A]{SimonGMT}), but we nevertheless repeat it here for the benefit of the reader.

\begin{proof}[Proof of \cref{thm:regularity-main}]
The smoothness of $u$ away from $\Sigma_u$ was established in
\cref{sec:eps-reg}. We first recall the low-dimensional rigidity used in
the reduction: every $C^{1,1}$ graphical special Lagrangian cone in $\C^n$ for $n\leq 4$ is a plane. This is a direct consequence of \cite[Lemma 2.1]{Yuan-3d-SL} when $n\leq 3$, and the main result of \cite{NadirashviliVladut2013} extends it to $n \leq 4$ for more general equations.

First of all, note that the works \cite{BW19_2d,Bh25} imply $\Sigma_u=\varnothing$ for $n\le4$ when the ambient space is $\C^n$. In a general almost K\"{a}hler manifold, we invoke \cref{prop:tangent-cones} to obtain a non-flat graphical special Lagrangian tangent cone with $C^{1,1}$ potential at a singular point $x_0\in\Sigma_u$, contradicting the above low-dimensional classification. We may therefore henceforth assume that $n\ge5$. 

Now, suppose for a contradiction that
$\dim_{\cH}\Sigma_u>n-5$. This in particular means that there exists $s>0$ such that
\[
    \Hcal^{n-5+s}_\infty(\Sigma_u) > 0\,,
\]
where we work with the Hausdorff content (see \cref{sec:notation}) in place of Hausdorff measure due to its upper semicontinuity under Hausdorff convergence. This in turn implies that there exist $x_0\in\Sigma_u$,
$c>0$, and radii $r_j\downarrow0$ such that
\[
  \cH^{n-5+s}_\infty\bigl(\Sigma_u\cap B_{r_j}(x_0)\bigr)
  \ge c r_j^{n-5+s}\,.
\]
After passing to a subsequence, the closed sets $(\Sigma_u)_{x_0,r_j} = \eta_{x_0,r_j}(\Sigma_u)$ converge locally in Hausdorff distance to a closed set $\Sigma_{\infty} \subset \R^n$, and the lower content bound passes to the limit, yielding
\[
  \cH^{n-5+s}_\infty(\Sigma_\infty\cap\overline B_1)\ge c\,,
\]
which in particular implies that $\dim_{\cH}\Sigma_\infty\ge n-5+s$.
Setting $q_0 := (x_0, Du(x_0))$, we may apply \cref{prop:tangent-cones} to the rescalings $T_j = (\eta_{q_0, r_j})_\sharp T$ for the canonical multiplicity-one integral current $T$ induced by the graph of $D u$. This gives subsequential local weak-$*$ convergence of $T_j$ to a multiplicity-one conical current $C$ induced by a Lipschitz graph, which is special Lagrangian in the Hermitian space $(\R^{2n}, g(q_0), J(q_0))$ isomorphic to the standard $\C^n$. In addition, $ \Sigma_\infty \subset \pi_{\R^{n} \times \{0\}}(\Sing C)$, so
\[
\dim_\Hcal \Sing C \geq n - 5 + s.
\]

We thus deduce that there exists a point $Y \in \Sing C\setminus \{0\}$, $c_1>0$ and radii $\rho_j \downarrow 0$ with
\begin{equation}\label{e:lower-content-bd-step1}
    \cH^{n-5+s}_\infty\bigl(\Sing C \cap B_{\rho_j}(Y)\bigr)
  \ge c_1 \rho_j^{n-5+s}\,.
\end{equation}
Taking a tangent cone to $C$ at $Y$ along $\rho_j$ and applying \cref{lem:graphical-splitting} yields a non-flat multiplicity one graphical special Lagrangian tangent cone  $ C_{Y} = \Span\{ Y\} \times C_0$ with the cross-section $C_0$ remaining a non-flat multiplicity-one graphical special Lagrangian cone in some $\C^{n-1}$. The lower bound \eqref{e:lower-content-bd-step1} combined with persistence of singularities for multiplicity-one area-minimizing currents yields
\[
    \cH^{n-5+s}_\infty(\Sing C_Y\cap\overline B_1) > 0\,,
    \quad \text{so} \quad
    \cH^{n-6+s}_\infty(\Sing C_0\cap\overline B_1) > 0\,.
\]
Iterating this, we obtain a non-flat multiplicity-one graphical special Lagrangian cone $\widehat C$ in $\C^5$ with $\cH^{s}_\infty(\Sing \widehat C\cap\overline B_1) > 0$. Thus, we can still find $Y\in \Sing \widehat C\setminus \{0\}$ and iterate one final time to obtain a non-flat multiplicity-one graphical special Lagrangian cone $C'$ in $\C^4$ contradicting the low-dimensional rigidity recalled at the beginning of the proof.

It remains to verify that $\Sigma_u$ is discrete when $n=5$. Suppose for a contradiction that there is a sequence of points $x_j \in \Sigma_u$ accumulating to $x_0$, and let $r_j:=2|x_j-x_0|$. After passing to
a subsequence,
\[
  (\Sigma_u)_{x_0,r_j} \ni y_j:=\frac{x_j-x_0}{r_j}\longrightarrow y_\infty\,,
  \qquad |y_\infty|=\frac{1}{2}\,,
\]
and as above, \cref{prop:tangent-cones} yields a limiting conical special Lagrangian current $C = \llbracket \graph(Dw)\rrbracket$ with a singular point $Y=(y_\infty, Dw(y_\infty)) \neq 0$ in $\R^{10}$ with respect to some complex structure compatible with $\omega_0$. Since $Y \neq 0$ and $C$ is a cone, \cref{lem:graphical-splitting} allows us to take a further tangent cone $C_Y = \mathrm{span}\{Y\} \times C_0$ at $Y$ with a singular four-dimensional cross-section $C_0$, reaching the same contradiction.
\end{proof}

\section{Construction of a singular graphical special Lagrangian cone}\label{sec:counterexample}
This section is dedicated to the proof of \cref{thm:counterexample-main}, that is, the construction of a singular special Lagrangian cone which is a graph over $\R^5$. Our solution $u \in C^{1,1}\setminus C^2(\R^5)$ to the zero-phase special Lagrangian equation is built as the $2$-homogeneous extension of a function $f \in C^\infty(\Sbb^4)$ which is constant on the leaves of Cartan's isoparametric foliation by hypersurfaces with three distinct principal curvatures.

Imposing this reduction leads to a second-order ODE for a function $f$ on the interval $[-1, 1]$ which is singular at $\pm 1$; however, an analytic solution exists around $1$ for every value of the initial condition $A = f(1)$ provided that $B = f'(1)$ is suitably chosen. Thanks to the symmetry of the problem, any solution $f_A$ on $[0, 1]$ such that $f_A(0) = 0$ can be oddly reflected to $[-1, 1]$, thus our goal becomes to find $A \in \R$ such that $\Phi(A) := f_A(0) = 0$.

\cref{fig:phase-zero-shooting-zoom} shows an unverified numerical plot of the behavior of $\Phi(A)$ for $0 \leq A < \sqrt{3}/2$; the case $A=0$ corresponds to a plane, whereas a zero appears to be found at $\widehat A \approx 0.669$ and the ODE becomes ill-posed for $A = \sqrt{3}/2$. Further, \cref{fig:plot-f} shows a numerical plot of $f_{\widehat A}$.

We rigorously verify these simulations in \cref{lem:shooting} by using validated numerics. The strategy is to implement a shooting method with interval arithmetic, in order to guarantee that the solution of the ODE exists on an interval $[A_-, A_+]$ and changes sign; see \cref{prop:validated-shooting} below for more details. The rest of the proof, which constitutes the body of this section, consists of algebraic manipulations and analytic arguments without computer involvement.

\begin{figure}[t]
  \centering
  \includegraphics[width=.88\linewidth]{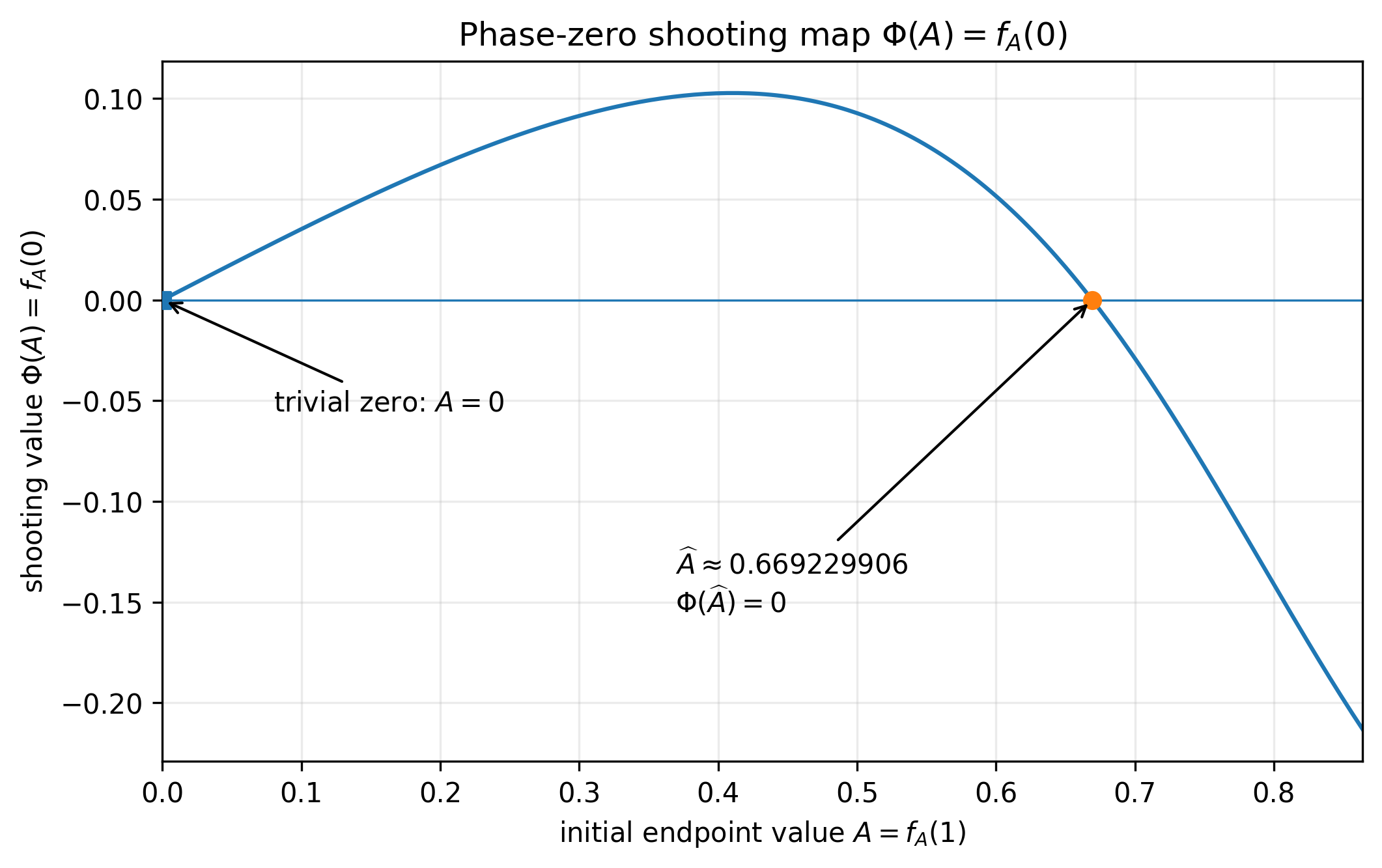}
  \caption{Numerical behavior of the endpoint value $f_A(0)$ for the solution $f_A$ to \eqref{eq:shooting-ODE} with initial data $f_A(1) = A$, $f_A'(1)= B(A)$, with $B(A)$ as in \eqref{eq:B-of-A} and $0 \leq A < \sqrt{3}/2$. This plot is just for illustration and not used in the proof.}
  \label{fig:phase-zero-shooting-zoom}
\end{figure}

\begin{figure}[t]
  \centering
  \includegraphics[width=.88\linewidth]{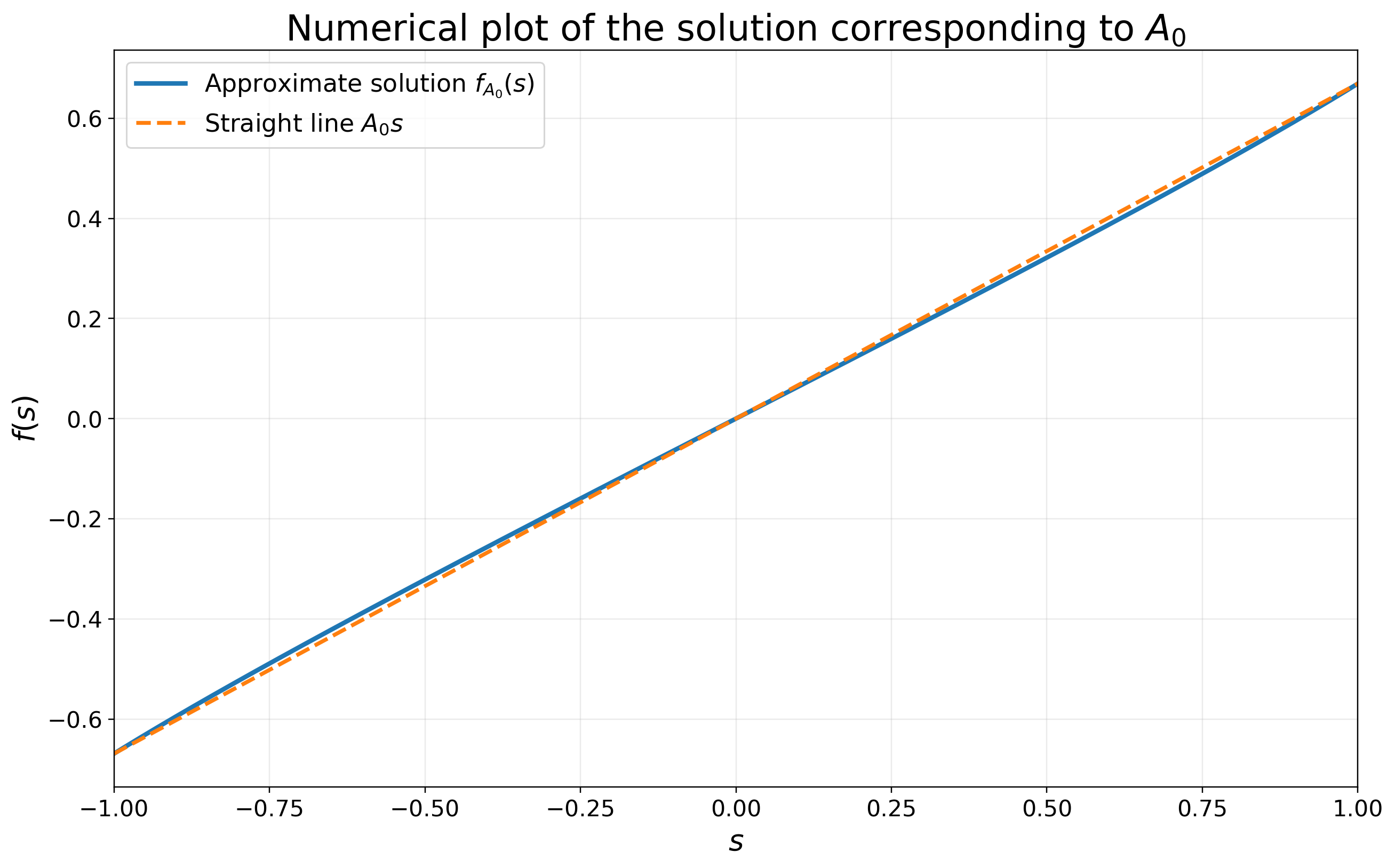}
  \caption{Numerical plot of the solution $f$ on $[-1, 1]$ for the parameter $A_0$ from \eqref{eq:Acenter}. A linear function is added in orange for comparison, given that the plot looks almost linear.}
  \label{fig:plot-f}
\end{figure}

\subsection{Cartan's cubic and orbit reduction}
In 1939, Cartan \cite{Cartan1939} discovered a foliation of $\Sbb^4$ by hypersurfaces with three distinct constant principal curvatures. These arise as the level sets of a cubic polynomial with interesting geometric and analytical properties, which have been used before to construct nontrivial solutions to geometric PDEs (see for example \cite{NadirashviliTkachevVladut2012, FiresterTsiamis26}).
Write
\[
  x=(x_1,x_2,z_1,z_2,z_3)\in\R^5,
  \qquad r=|x|,
\]
and define Cartan's isoparametric cubic
\begin{equation}\label{eq:Cartan-cubic}
  P(x)=x_1^3+\frac{3x_1}{2}
  \bigl(z_1^2+z_2^2-2z_3^2-2x_2^2\bigr)
  +\frac{3\sqrt3}{2}
  \bigl(x_2z_1^2-x_2z_2^2+2z_1z_2z_3\bigr).
\end{equation}
Direct differentiation yields
\begin{equation}\label{eq:Cartan-identities}
  |\nabla P|^2=9r^4,
  \qquad \Delta P=0.
\end{equation}
One obtains a convenient description of $P$ by identifying $\R^5$ with the space of traceless symmetric matrices via the map
\begin{equation*}\label{eq:Cartan-Q}
Q(x)=
\begin{pmatrix}
2x_1&\sqrt3z_1&\sqrt3z_2\\
\sqrt3z_1&-x_1+\sqrt3x_2&\sqrt3z_3\\
\sqrt3z_2&\sqrt3z_3&-x_1-\sqrt3x_2
\end{pmatrix}.
\end{equation*}
By expanding we see that
\begin{equation*}\label{eq:Q-invariants}
  P(x)=\frac12\det Q(x),
  \qquad
  r^2=\frac16\tr(Q(x)^2),
\end{equation*}
and consequently, the conjugation action $Q\mapsto RQR^T$, $R\in \mathrm{SO}(3)$, induces an orthogonal action on $\R^5$ preserving both $r$ and $P$.  Conversely, two traceless symmetric $3\times3$ matrices with the same values of $\tr(Q^2)$ and $\det Q$ have the same characteristic polynomial and therefore are conjugate by a matrix which, multiplying by $-1$ if necessary, we may assume that lies in $\mathrm{SO}(3)$.  Hence the $\mathrm{SO}(3)$-orbits on $\Sbb^4$ are precisely the level sets of $P$.

Let $s(x) := r^{-3} P(x)$ be the zero-homogeneous extension of  $P$ to $\R^5 \setminus \{ 0 \}$. Restricting again to $\Sbb^4$, Euler's identity and \eqref{eq:Cartan-identities} give
\[
  |\nabla_{\Sbb^4}s|^2=9(1-s^2),
  \qquad
  \Delta_{\Sbb^4}s=-18s.
\]
Thus the level sets $M_s$ of $s \neq \pm 1$ form an isoparametric (and in particular CMC) family in $\Sbb^4$. Notice that $s$ is a smooth function on $\Sbb^4$ and its only critical points are along the focal submanifolds $M_{\pm 1}$, which consist each of a real Veronese surface $\mathbb{RP}^2 \subset \Sbb^4$.

Since every symmetric matrix can be orthogonally diagonalized, every orbit meets the two-dimensional linear subspace $\pi:=\{z_1=z_2=z_3=0\}$.  Permutations of the three diagonal entries are realized by conjugation with signed permutation matrices in $\mathrm{SO}(3)$, so a fundamental Weyl chamber is represented by the arc of $\Sbb^4\cap\pi\cong\Sbb^1$ on which
\[
  x_1=\cos\theta,
  \qquad x_2=\sin\theta,
  \qquad 0\le\theta\le\frac\pi3.
\]
Then
\begin{equation*}\label{eq:s-section}
  s=\cos^3\theta-3\cos\theta\sin^2\theta=\cos3\theta,
  \qquad
  \sin3\theta=\sqrt{1-s^2}\ge0.
\end{equation*}
This arc meets every orbit exactly once and $s$ decreases monotonically from $1$ to $-1$ along it.  In particular, the range of $s$ on $\Sbb^4$ is $[-1,1]$.

We seek an $\mathrm{SO}(3)$-invariant $2$-homogeneous potential
\begin{equation}\label{eq:Cartan-ansatz}
  u(x)=r^2f(s(x)),
  \qquad u(0)=0,
\end{equation}
where $f$ is a real-valued function on $[-1,1]$, {smooth up to the boundary (in fact in our construction it will be real analytic)}.

\subsection{Hessian and determinant reduction}\label{ss:Hessian-and-det-reduction}
We now turn to analyzing the special Lagrangian equation
\begin{equation}
\label{eq:special-lagrangian}
\Ima \det (\Id + i D^2 u) = 0
\qquad \text{on }
\R^5 \setminus \{0\}
\end{equation}
for the Ansatz \eqref{eq:Cartan-ansatz}. As $D^2 u$ is $0$-homogeneous, we may restrict to $\Sbb^4$. Moreover, since the $\mathrm{SO}(3)$-action is orthogonal and preserves $u$, the eigenvalues of $D^2u$ are constant along each orbit, thus it suffices to compute them at a point $(\cos \theta, \sin \theta, 0, 0, 0)$ with $\theta \in [0, \pi/3]$. At this point, we will use the adapted orthonormal basis
\[
\begin{split}
  e_r=(\cos\theta,\sin\theta,0,0,0),
  \qquad e_\theta=(-\sin\theta,\cos\theta,0,0,0) \\
  e_1 = (0, 0, 1, 0, 0), \qquad
  e_2 = (0, 0, 0, 1, 0), \qquad
  e_3 = (0, 0, 0, 0, 1)
\end{split}
\]
and note that $\nabla s=-3\sin3\theta\,e_\theta$, as one can compute directly. We also have
\begin{align*}
  \nabla s&=\nabla P-3P r^{-5} x,\\
  D^2s&=D^2P-3(x\otimes\nabla P+\nabla P\otimes x)
       -3s\,\Id+15s\,x\otimes x
\end{align*}
on $\Sbb^4$; using the explicit expression \eqref{eq:Cartan-cubic}, we get
\[
D^2 P =
\begin{pmatrix}
    6 \cos 3\theta & -6 \sin 3 \theta & 0 & 0 & 0 \\
    -6 \sin 3 \theta & -6 \cos 3 \theta & 0 & 0 & 0 \\
    0 & 0 & 3 (\cos \theta + \sqrt{3} \sin \theta) & 0 & 0 \\
    0 & 0 & 0 & 3 (\cos \theta - \sqrt{3} \sin \theta) & 0 \\
    0 & 0 & 0 & 0 & -6 \cos \theta
\end{pmatrix}
.
\]
Thus
\begin{equation}
\label{eq:hessian-s}
D^2 s =
\begin{pmatrix}
    0 & 3 \sin 3 \theta & 0 & 0 & 0 \\
    3 \sin 3 \theta & -9 \cos 3 \theta & 0 & 0 & 0 \\
    0 & 0 & \mu_1 & 0 & 0 \\
    0 & 0 & 0 & \mu_2 & 0 \\
    0 & 0 & 0 & 0 & \mu_3
\end{pmatrix}
,
\end{equation}
where we have introduced
\begin{equation*}\label{eq:mu-explicit}
  \mu_1 := 3(\cos\theta+\sqrt3\sin\theta-s),
  \quad
  \mu_2 := 3(\cos\theta-\sqrt3\sin\theta-s),
  \quad
  \mu_3 := -3(2\cos\theta+s).
\end{equation*}
Although we will not use this, we remark that $\operatorname{diag}(\mu_1, \mu_2, \mu_3) = D^2 s |_{T M_s} = -A_{M_s} \cdot \nabla s = -|\nabla s| \operatorname{diag}(\kappa_1, \kappa_2, \kappa_3)$, where $\kappa_\ell = \cot(\theta - \ell \pi / 3)$, $\ell = 1,2,3$ are the principal curvatures of $M_s$.
Differentiating $u = r^2 (f \circ s)$ and substituting into \eqref{eq:hessian-s} gives
\begin{equation*}
\label{eq:Hessian-block}
\begin{aligned}
D^2u
&=2f\Id+2f'(x\otimes\nabla s+\nabla s\otimes x)
+f''\,\nabla s\otimes\nabla s + f' D^2s \\
&=
\begin{pmatrix}
2f&-3f'\sin3\theta\\
-3f'\sin3\theta&2f+9(1-s^2)f''-9sf'
\end{pmatrix}
\oplus\operatorname{diag}(2f+f'\mu_1,2f+f'\mu_2,2f+f'\mu_3).
\end{aligned}
\end{equation*}
Next we compute $\det (\Id + i D^2 u)$; for that, it is useful to first obtain the elementary symmetric functions of $\mu_1,\mu_2,\mu_3$, which turn out to be polynomials in $s$:
\begin{equation*}\label{eq:mu-symmetric}
  \sum_{j=1}^3\mu_j=-9s,
  \qquad
  \sum_{k<j}\mu_k\mu_j=-27(1-s^2),
  \qquad
  \mu_1\mu_2\mu_3=27s(1-s^2).
\end{equation*}
Introducing $K:=1+2if$, the determinant of the first $2\times2$ block of $\Id+iD^2u$ is
\[
  K^2+(-3f'\sin3\theta)^2
  +9iK\bigl((1-s^2)f''-sf'\bigr),
\]
whereas the determinant of the $3\times3$ block along the directions tangent to $M_s$ is
\begin{equation*}
\label{eq:T-product}
\begin{aligned}
  T &:=\prod_{j=1}^3(K+if'\mu_j)
    = K^3 + K^2 if' \sum_{j} \mu_j - K (f')^2 \sum_{j<k} \mu_j \mu_k - i (f')^3 \mu_1\mu_2\mu_3 \\
    &= K^3 + 27 K (f')^2 (1-s^2) - 9i (K^2 f' s + 3 (f')^3 s(1-s^2)).
\end{aligned}
\end{equation*}
Therefore, using $\sin^2 (3\theta) = 1-s^2$, we see that $\det(\Id + i D^2 u)$ is a polynomial in $s, f, f'$ and $f''$:
\begin{equation}\label{eq:det-reduction}
  \det(\Id+iD^2u)
  =T\left[K^2+9(f')^2 (1-s^2) +9iK\bigl((1-s^2)f''-sf'\bigr)\right].
\end{equation}
Next we will proceed to impose \eqref{eq:special-lagrangian}, and isolate $f''$. With this in mind, we introduce the following polynomial expressions on $s, f$ and $f'$:
\begin{equation}\label{eq:DN-def}
  D := \Rea(KT),
  \qquad
  N := \Ima\bigl((K^2+9(1-s^2)(f')^2)T\bigr),
\end{equation}
so that the imaginary part of \eqref{eq:det-reduction} is
\begin{equation}\label{eq:imag-det-reduced}
  \Ima\det(\Id+iD^2u)
  =N+9\bigl((1-s^2)f''-s f'\bigr)D.
\end{equation}
Thus wherever $D\ne0$ and $|s|<1$, the special Lagrangian equation \eqref{eq:special-lagrangian} yields
\begin{equation}\label{eq:shooting-ODE}
  f''=\frac{9sf'-N/D}{9(1-s^2)}.
\end{equation}

Explicit formulas for $D(s, f, y)$ and $N(s, f, y)$, where we are substituting $y=f'$, are recorded in \cref{app:polynomials}.  Note, in particular, that these are polynomials with integer coefficients.

\subsection{\texorpdfstring{Local existence and uniqueness near $s=1$}{Local existence and uniqueness near s=1}} \label{local_existence}
Now we verify that the ODE \eqref{eq:shooting-ODE} admits a unique solution in the analytic class locally around $s=1$, for suitable boundary data $A = f(1)$ and $B = f'(1)$. This local existence and uniqueness will be the starting point for the shooting procedure which we will subsequently use to obtain a global solution.

Since the denominator of \eqref{eq:shooting-ODE} vanishes at $s = 1$, we must at least impose that the numerator vanishes there too, or equivalently, that \eqref{eq:imag-det-reduced} is zero for $s = 1$, to avoid singular behavior at this endpoint. At $s = 1$, we can compute explicitly $K = 1+2iA$ and $T = (1+2iA)^2 (1+2iA-9iB)$, which in turn implies that \eqref{eq:det-reduction} simplifies to 
\begin{equation}\label{eq:Hessian-det-s=1}
    \det(\Id + iD^2 u) = (1+2iA)^3(1+2iA-9iB)^2\,.
\end{equation}
Requiring this quantity to be real is in turn implied by the argument condition $\arg(KT) + \arg(1+2iA-9iB) =0$, which amounts to
\begin{equation}
    \label{eq:argument-at-start}
    3 \arctan (2A) + 2 \arctan(2A - 9B) = 0\,.
\end{equation}
Below we fix an interval $[A_-, A_+] \subset (-\sqrt{3}/2, \sqrt{3}/2)$ (see \eqref{eq:Aminus}-\eqref{eq:Aplus}), in which we will proceed to obtain existence of an analytic solution. In such an interval, \eqref{eq:argument-at-start} can be solved analytically by
\begin{equation}\label{eq:B-of-A}
  B=B(A):=\frac{2A+\tan\!\left(\frac32\arctan(2A)\right)}9\,.
\end{equation}

We next verify that
\begin{equation}\label{eq:D-endpoint-positive}
  D(1,A,B(A))>0\,,
\end{equation}
which will ensure that the reduced equation for $f''$ in \eqref{eq:shooting-ODE} does not have further poles at $(1, A, B(A))$. By the expressions for $K$ and $T$ found above, we have $KT = (1+2iA)^3 (1+2iA-9iB)$. Thanks to \eqref{eq:argument-at-start} and $|2A| < \sqrt{3}$, the argument modulo $2 \pi \Z$ of $KT$ is therefore
\[
  3 \arctan(2A) + \arctan(2A-9B)
  = \frac{3}{2} \arctan(2A)
  \in \left( -\frac{\pi}{2}, \frac{\pi}{2} \right)\,,
\]
so \eqref{eq:D-endpoint-positive} is indeed true.

Next, let us rewrite the ODE in the form of a Briot-Bouquet type system. Set $y(s) = f'(s)$, $t=1-s$ and use a dot for differentiation with respect to $t$. Equation \eqref{eq:shooting-ODE} then becomes the system
\begin{equation}\label{eq:BB-system}
  t \Dot{f}(1-t)=-ty,
  \qquad
  t\Dot{y}(1-t)=\mathcal G(t,f(1-t),y(1-t)),
\end{equation}
where
\begin{equation}\label{eq:G-definition}
  \mathcal G(t,f(1-t),y(1-t))
  :=\frac{N(1-t,f,y)-9(1-t)yD(1-t,f,y)}
  {9(2-t)D(1-t,f,y)}.
\end{equation}
We will henceforth simply write $(t,f,y)$, suppressing the implicit fact that we are evaluating $(f,y)$ at $1-t$. At $(t,f,y)=(0,A,B(A))$, \eqref{eq:argument-at-start} guarantees that the numerator in \eqref{eq:G-definition} vanishes and the denominator does not by \eqref{eq:D-endpoint-positive}, so $\mathcal G=0$.  To compute its linearization, observe that at $t=0$, the identity \eqref{eq:det-reduction} simplifies to \eqref{eq:Hessian-det-s=1}, so by \eqref{eq:imag-det-reduced} the numerator of $\mathcal{G}$ becomes
\begin{equation}
\label{eq:numerator-G}
  N(1,f,y)-9yD(1,f,y)
  =\Ima \left[ (1+2if)^3(1+2if-9iy)^2 \right].
\end{equation}
Differentiating with respect to $y$ and using that \eqref{eq:numerator-G} vanishes at $(t, f, y) = (0, A, B(A))$, together with \eqref{eq:DN-def}, gives
\begin{equation}
\label{eq:G-y-minus-one}
\begin{split}
  \partial_y \mathcal{G}(0, A, B(A))
  &= \frac{-18 \Rea \left[ (1+2iA)^3 (1+2iA-9iB(A)) \right]}{9(2-0)D(1, A, B(A))} \\
  &= -\frac{18 \Rea KT}{18 D(1, A, B(A))}
  = -1\,.
\end{split}
\end{equation}
{Note that we do not specifically require the right-hand side of \eqref{eq:G-y-minus-one} to be $-1$; any value (which is allowed to depend on $A$) that yields a suitable bound on $M_y$ in Subsection \ref{ss:ODE-shooting} would suffice. This particularly clean value is merely a consequence of the symmetry of our ansatz.

Writing $Z=(f,y)$ and $\Phi(t,Z)=(-ty,\mathcal G(t,f,y))$, the Jacobian
$J_A:=D_Z\Phi(0,A,B(A))$ is
\begin{equation*}
  J_A=
  \begin{pmatrix}
  0&0\\
  \mathcal \partial_f \mathcal{G}(0,A,B(A))&-1
  \end{pmatrix};
\end{equation*}
in particular, its eigenvalues are $0$ and $-1$.

Our key local analytic existence result near $t=0$ is the following.
\begin{lemma}[Local analytic existence]\label{lem:endpoint-family}
Let $I\Subset(-\sqrt3/2,\sqrt3/2)$ be a compact interval. Then there exists $\eps > 0$ such that for every $A \in I$, the ODE \eqref{eq:BB-system} has a unique real-analytic solution over $t\in [0,\eps]$ satisfying
\[
  {(f, y)|_{t = 0} = (A, B(A)).}
\]
Moreover, the solution and all of its Taylor coefficients depend real-analytically on $A$.
\end{lemma}

\begin{proof}
We seek a solution $Z(t)=(f(t),y(t))=(A,B(A))+W(t)$ that is a perturbation of our initial condition, where $W$ will be determined via an analytic fixed point argument.  By \eqref{eq:D-endpoint-positive}, the map $\Phi$ is real analytic in a neighborhood of
$\{(0,A,B(A)):A\in I\}$.  The system \eqref{eq:BB-system} turns into the following ODE for $W$:
\begin{equation}\label{eq:BB-shifted}
  t\dot W-J_AW=\Psi_A(t,W)\,,
\end{equation}
where $\Psi_A(t,W):=\Phi(t,(A,B(A))+W)-J_AW$, with the initial conditions $\Psi_A(0,0)=0$ and $ D_W\Psi_A(0,0)=0$.

For $\rho>0$, let $X_\rho$ be the Banach space of vector-valued power series
$W(t)=\sum_{k\ge1}W_kt^k$ defined over the $t$-interval $[0,\rho]$, equipped with norm
\[
  \|W\|_\rho:=\sum_{k\ge1}|W_k|\rho^k.
\]
On series $\sum_{k\geq 1} H_k t^k$ with zero constant term, the linear operator
$\mathcal L_AW := t\dot W-J_AW$ has the formal inverse
\[
  \mathcal L_A^{-1}\!\left(\sum_{k\ge1}H_kt^k\right)
  =\sum_{k\ge1}(k\Id-J_A)^{-1}H_k t^k,
\]
which is well-defined because no positive integer is an eigenvalue of $J_A$. Since $\partial_f \mathcal{G}(0,A,B(A))$ is uniformly bounded on $I$, we have
\[
  M:=\sup_{A\in I}\sup_{k\ge1}
  \left|(k\Id-J_A)^{-1}\right|<\infty
\]
in any matrix norm. Thus $\|\mathcal L_A^{-1}\|_{X_\rho\to X_\rho}\le M$, uniformly in $A$.

After complexifying the analytic functions, Cauchy estimates and the initial conditions on $\Psi_A$ give a constant $C$, independent of $A\in I$, such that for $W,\widetilde W$ in a sufficiently small ball within $X_\rho$,
\begin{align*}
  \|\Psi_A(\,\cdot\,,W)\|_\rho
  &\le C\bigl(\rho+\|W\|_\rho^2\bigr),\\
  \|\Psi_A(\,\cdot\,,W)-
     \Psi_A(\,\cdot\,,\widetilde W)\|_\rho
  &\le C\bigl(\rho+\|W\|_\rho+\|\widetilde W\|_\rho\bigr)
       \|W-\widetilde W\|_\rho.
\end{align*}
Setting $R=2MC\rho$ and choosing $\rho$ small enough, we deduce that the map
\[
  W\longmapsto\mathcal L_A^{-1}
  \Psi_A(\,\cdot\,,W)
\]
maps the closed radius $R$ ball of $X_\rho$ into itself and is a contraction there, uniformly for $A\in I$.  Its unique fixed point is a convergent power series solving \eqref{eq:BB-shifted}. Noting that everything is real-valued (the complexification was merely for the above estimates), this gives the claimed real-analytic solution on $t\in [0, \varepsilon]$, equivalently, $s\in [1-\eps, 1]$, with $\eps$ taken to be the upper threshold on $\rho$ obtained above.

The same contraction works after complexifying the initial data $A$ in small balls.  Picard iteration therefore shows that the fixed point is a holomorphic function in $A$; equivalently, comparison of coefficients gives
\[
  (k\Id-J_A)W_k=R_k(W_1,\ldots,W_{k-1};A),
  \qquad k\ge1\,,
\]
and determines each coefficient $W_k$ recursively from the preceding ones with real-analytic dependence.  Any other real-analytic solution with the same initial value belongs to the contraction ball after decreasing $\rho$, and hence coincides with this fixed point. This proves uniqueness and real-analytic parameter dependence.
\end{proof}

\begin{remark}\label{r:recursive-analytic-coeffs}
For future reference, if we write the local expansions
\begin{equation}
\label{eq:analytic-expansions-y-f}
  y(t)=B+\sum_{k\ge1}b_k t^k\,,
  \qquad
  f(t)=A-\sum_{k\ge1}\frac{b_{k-1}}k t^k\,,
\end{equation}
near $(t, f, y)=(0,A,B(A))$, then the order $k$ coefficient equation has the form
$(k+1)b_k=R_k(b_0,\ldots,b_{k-1};A)$. This will be used shortly to find local Taylor polynomial approximations for $(f,y)$ near the endpoint $t=0$, which will be the starting point for the shooting argument.
\end{remark}

\subsection{\texorpdfstring{Global solution on $[0,1]$ and endpoint signs}{Global solution on [0,1] and endpoint signs}} \label{global_sol}
We now explain the computer-assisted argument to extend the local existence of Lemma \ref{lem:endpoint-family} to global existence on the $s$ parameter interval $[0,1]$ and control the sign of $f_{A_\pm}(0)$. The narrow initial data interval {$[A_-,A_+]= [A_0-10^{-26},A_0+10^{-26}]$} below was chosen a posteriori to obtain a high-precision numerical approximation $A_0$ of a zero of the shooting map $A\mapsto f_A(0)$. We will prove that an analytic solution exists for $A \in [A_-, A_+]$ and $s \in [0, 1]$, with opposite signs at $s = 0$ when $A = A_\pm$, but any initial data interval with these properties would suffice for the global existence argument.

Define the parameters
\begin{align}
A_-&:=0.66922990609204402834133930,\label{eq:Aminus}\\
A_0&:=0.66922990609204402834133931, \label{eq:Acenter} \\
A_+&:=0.66922990609204402834133932.\label{eq:Aplus}
\end{align}

\begin{proposition}[Global existence]\label{prop:validated-shooting}
For every $A\in[A_-,A_+]$, the local analytic solution from \cref{lem:endpoint-family} extends uniquely to $0\le t\le1$, and the solutions associated to the endpoints satisfy
\begin{align}
    f_{A_-}(0) &\in
 \bigl[7.27084429, 7.27084453\bigr] 
 \times10^{-27},\label{eq:positive-shot}\\
 f_{A_+}(0) &\in
 \bigl[-1.04396392, -1.04396389\bigr]
 \times10^{-26}.\label{eq:negative-shot}
\end{align}
In particular, the two values have opposite signs.
\end{proposition}

The proof of \cref{prop:validated-shooting} is computer-assisted, using interval arithmetic from the Python libraries \texttt{FLINT} and \texttt{Arb}. 
We first describe the logic of the argument and prove the proposition assuming that certain parameters, which we fix a priori, and certain quantities, which depend on them, satisfy a list of constraints. The precise values of the prescribed parameters are provided in a table in Subsection \ref{sss:param} below. Enclosures for the remaining quantities are computed by the validating program, which also checks that all constraints are satisfied. {We exhibit below upper bounds for some of the most important quantities for illustration.}

Recalling that we are working with the system \eqref{eq:BB-system}, we make the following separate claims: the ODE is well-posed in $t \in [0, 1]$ for $A \in [A_-, A_+]$, and for $A = A_\pm$ its solution at $t = 1$ has sign $\mp$. In order to verify these claims, it is useful to treat them simultaneously and get the best possible enclosures for $(f_A, y_A)$ uniformly in $A \in [A_c - A_r, A_c + A_r]$ in the three cases 
\begin{equation}\label{eq:initial-data-cases}
    (A_c, A_r) = \begin{cases}
        (A_-, 0) \\
        (A_+, 0) \\
        \left( \tfrac{A_- + A_+}{2}, \tfrac{A_+ - A_-}{2} \right)\,,
    \end{cases}
\end{equation} 
that is, treating the parameter $A$ as an interval. Showing good enough enclosures for $(f_A, y_A)$ in successive time subintervals will, in turn, show that the solution continues to exist for further times, and ultimately conclude that for all $A \in [A_-, A_+]$ it extends to $t = 1$, together with the desired sign condition.

Due to the singular nature of the ODE, as in the analytic proof from above, we split the interval $[0, 1]$ into $[0, h_0]$ and $[h_0, 1]$ but this time for a chosen macroscopic parameter $h_0$. In the first interval, we propose a polynomial approximate solution to the ODE constructed by solving recursively for the coefficients \eqref{eq:analytic-expansions-y-f}, and show by a fixed-point argument that the actual solution lies in a controlled $C^0$-neighborhood of it. On $[h_0, 1]$, we similarly construct polynomial ans\"atze in small subdivisions of the time interval, and propagate the deviation from the approximate solution by a Gr\"{o}nwall argument.

\smallskip
\subsubsection{Validation of the singular ODE on $[0, h_0]$.}\label{ss:ODE-shooting}
We begin by computing with exact rational arithmetic {(i.e. with error zero)} the polynomials $p_f(t), p_y(t) \in \mathbb Q[t]$ which are the truncations (resp.~to a fixed degree $d_0+1$ and $d_0$) of the analytic solution constructed in \cref{lem:endpoint-family} with initial data $(A_c, B_c)$, where $B_c$ {is a very precise rational approximation of} $B(A_c)$. Thus they satisfy
\[
  p_f(0) = A_c,
  \qquad p_y(0) = B_c
  \qquad \dot p_f=-p_y\,,
\]
and their coefficients are calculated recursively as explained in \cref{r:recursive-analytic-coeffs}.
The precise details of the construction of this approximate solution has no bearing on the proof; the only property of $(p_f, p_y)$ that we use is its initial condition and the fact it solves \eqref{eq:BB-system} approximately, in the sense that the error
\begin{equation}\label{e:error-polynom}
  \mathcal R_0(t):=t \Dot{p}_y(t)-\mathcal G(t,p_f(t),p_y(t))
\end{equation}
has a sufficiently small $C^0$ norm. Practically, the bound $\rho_0  := \| \mathcal{R}_0 \|_{C^0([0, h_0])}$ can be computed from an enclosure of $\mathcal{R}_0([0, h_0])$. In our case, $\mathcal{R}_0$ can be written as the quotient of two polynomials: the denominator is comfortably far from zero, and the numerator has all coefficients of degree $\leq d_0$ vanishing (up to a numerical error) by construction. Thanks to this we get the tiny bound \eqref{eq:CAP-rho0-Mf-My} for $\rho_0 $.

We proceed to set up a fixed point method to compare the exact local solution of \eqref{eq:BB-system} to these polynomial approximations. 
For $|A-A_c|\le A_r$, seek a solution $(f,y)$ of the form
\begin{equation}\label{e:exact-sol}
  f=p_f+u\,,
  \qquad y=p_y+v\,,
  \qquad u(t)=A - A_c -\int_0^t v(\tau)\,\dd \tau\,.
\end{equation}
Then, thanks to \eqref{e:error-polynom}, the equation for $\dot{y}$ is equivalent to $v$ being a fixed point of the mapping $F : \mathbf{B}_{\overline{V}} \subset C^0([0, h_0]) \to C^0([0, h_0])$ given by
\[
  Fv(t) := \frac1t\int_0^t
  \left[
  \mathcal G\bigl(\tau,(p_f+u)(\tau),(p_y+v)(\tau)\bigr)-\mathcal G\bigl(\tau,p_f(\tau),p_y(\tau)\bigr){+v}(\tau)-\mathcal R_0(\tau)
  \right]\dd \tau\,,
\]
with $u$ determined by \eqref{e:exact-sol}; note that $Fv$ indeed extends continuously to $t = 0$ as
\begin{equation}
    \label{eq:Fv-0}
    Fv(0) = \mathcal{G}(0, A, p_y(0) + v(0)) + v(0).
\end{equation}

We have restricted $F$ to a ball of fixed radius $\overline{V}$ in order to guarantee that $\mathcal{G}$ is well-behaved on its input. More precisely, first notice that $\|v\|_{C^0} \leq \overline{V}$ implies $\|u\|_{C^0} \leq \overline{U} := A_r + h_0 \overline{V}$. We define
\[
  M_f := \sup \left\{ |\partial_f \mathcal{G}(t, p_f(t) + u, p_y(t) + v)| : (t, u, v) \in [0, h_0] \times [-\overline{U}, \overline{U}] \times [-\overline{V}, \overline{V}] \right\}
\]
and
\[
  M_y := \sup \left\{ |\partial_y \mathcal{G}(t, p_f(t) + u, p_y(t) + v) + 1| : (t, u, v) \in [0, h_0] \times [-\overline{U}, \overline{U}] \times [-\overline{V}, \overline{V}] \right\}
\]
and compute enclosures for both (recorded below in \eqref{eq:CAP-rho0-Mf-My}) by splitting this box into a prescribed number $N_0$ equal-length pieces along the $t$-direction and evaluating on these boxes. In particular, $\mathcal G$ is Lipschitz on this box, and thus $F$ is well-defined. In addition, a direct computation shows that $F$ is Lipschitz in $\mathbf{B}_{\overline{V}}$ with constant $\Lip(F) \leq q := h_0 M_f + M_y$, and the value of $q < 1$ is also stated below in \eqref{eq:CAP-q}.

Next we seek a fixed point in $\mathbf{B}_V \subset C^0([0, h_0])$ for $V \leq \overline{V}$ as small as possible. For $v \equiv 0$ we have $u \equiv A - A_c$ and thus $\|F(0)\|_{C^0([0,h_0])} \leq M_f A_r + \rho_0 $, hence
\[
  \| F (v) \|_{C^0([0, h_0])}
  \leq \| F (0) \|_{C^0([0, h_0])} + \Lip(F) \| v \|_{C^0([0, h_0])}
  \leq M_f A_r + \rho_0  + q V
  = V
\]
provided that we set $V := \tfrac{M_f A_r + \rho_0 }{1 - q}$
and this number is smaller than $\overline{V}$; this is clear from the bounds obtained in \eqref{eq:CAP-V} below. The fixed point theorem now yields $v \in C^0([0, h_0])$ with norm at most $V$ such that for the corresponding $u$, the pair $(f, y) = (p_f + u, p_y + v)$ solves \eqref{eq:BB-system} with ${f|_{t=0}} = A$ and, by \eqref{eq:Fv-0}, $\mathcal{G}(0, A, {y|_{t=0}}) = 0$. {Moreover,} the only solution ${y|_{t=0}}$ to this equation in the $V$-neighborhood of $B_c$ is
${y|_{t=0}} = B(A)$.

By uniqueness of the fixed point (possibly in a shorter interval), this solution agrees with the analytic solution constructed in \cref{lem:endpoint-family}. Thus, the conclusion of this stage is that for all $A$ such that $|A - A_c| \leq A_r$, the solution from \cref{lem:endpoint-family} exists until $t = h_0$ and satisfies $|f(h_0) - f_1| \leq R_1$ and $|y(h_0) - y_1| \leq R_1$, where $(f_1, y_1) = (p_f(h_0), p_y(h_0))$ and $R_1 = \max \{ A_r + h_0 V, V \}$.

Here we record the bounds for $\rho_0 , M_f$ and $M_y$ produced by the validator, and the corresponding bound for $q$, for the run on $[A_-, A_+]$, since this gives the crudest bounds among the three cases \eqref{eq:initial-data-cases}.
\begin{equation}
\label{eq:CAP-rho0-Mf-My}
  \rho_0 \le 5.61 \times 10^{-46}, \qquad
  M_f \le 4.3749, \qquad
  M_y \le 0.1477, \qquad
\end{equation}
\begin{align}
  \label{eq:CAP-q} \Lip(F) \le q &\le 0.1915 <1.
\end{align}
The bounds for $V$ are as follows:
\begin{equation}
  \label{eq:CAP-V}
  V \leq 6.93 \times 10^{-46} \text{ for the endpoint runs,}
  \qquad
  V \leq 5.42 \times 10^{-26} \text{ for the interval run.}
\end{equation}

\smallskip
\subsubsection{Validation of the regular ODE on $[h_0, 1]$}
Next we explain how one carries out simultaneously the verification that $z(t) = (f(t), y(t))$-the solution with initial condition $A \in [A_c - A_r, A_c + A_r]$-extends to the $t$-interval $[h_0, 1]$, and the propagation of the enclosures for $z$ up to $t = 1$. We divide $[h_0, 1]$ into subintervals $[t_i, t_{i+1}]$ of variable length
\[
  h_i = t_{i+1} - t_i := \min\{t_i/4, \overline{h}, 1 - t_i\},
  \qquad t_1 = h_0\,,
\]
for a prescribed $\overline{h}$. In each subinterval, we seek a solution $z$ of the ODE system
\begin{equation}\label{eq:regular-vector-field}
  \dot z(t) =\mathcal V(t,z(t)),
  \qquad
  \mathcal V(t,f,y):=
  \left(-y,\frac{\mathcal G(t,f,y)}t\right),
  \qquad t \in [t_i, t_{i+1}],
\end{equation}
with initial conditions $z(t_i) = (f(t_i), y(t_i))$ satisfying the validated bounds $|f(t_i) - f_i| \leq R_i$, $|y(t_i) - y_i| \leq R_i$. As in the initial step, we find polynomials $p = (p_f, p_y)$ corresponding to the truncation to degrees $(d_1 + 1, d_1)$ of the power series solution of \eqref{eq:regular-vector-field} around $t_i$ with initial conditions $p(t_i) = (f_i, y_i)$. Analogously to the initial step, these solve \eqref{eq:regular-vector-field} up to an error
\[
  \mathcal{R}_i(t) := \dot p_y(t) - \frac{\mathcal{G}(t, p_f(t), p_y(t))}{t},
  \qquad
  \rho_i := \| \mathcal{R}_i \|_{C^0([t_i, t_{i+1}])}\,,
\]
which we estimate by lower-bounding the denominator that appears, and upper-bounding the numerator, which is a polynomial with minuscule terms of degree at most $d_1$.

Fix a coarse radius $\overline{R}$ and suppose that $\overline{R} > R_i$. As in the initial step, we can compute an enclosure for
\[
  L_i := \max \left\{ |D_z \mathcal{V}(t, p_f(t) + u, p_y(t) + v)| : (t, u, v) \in [t_i, t_{i+1}] \times [-\overline{R}, \overline{R}]^2 \right\},
\]
(where the matrix norm is $\ell^\infty \to \ell^\infty$) by further splitting this box in the $t$-coordinate into $N_1$ equal-length pieces and using interval arithmetic.

Let $w(t) := |z(t) - p(t)|$ using the $\ell^\infty$ norm on $\R^2$, and estimate
\[
  \frac{\dd}{\dd t} w(t)
  \leq |\dot z(t) - \dot p(t)|
  \leq |\mathcal{V}(t, z(t)) - \mathcal{V}(t, p(t))| + \rho_i
  \leq L_i w(t) + \rho_i
\]
as long as $w(t) \leq \overline{R}$. For such $t$, Gr\"{o}nwall's inequality gives
\[
  \frac{\dd}{\dd t} (e^{-L_i(t-t_i)} w) = e^{-L_i(t-t_i)} (\dot w - L_i w) \leq e^{-L_i(t-t_i)} \rho_i \leq \rho_i,
\]
and a first exit time argument shows that $w(t) \leq e^{L_i h_i} (R_i + \rho_i h_i) =: R_{i+1}$ for $t \in [t_i, t_{i+1}]$ provided that $R_{i+1} \leq \overline{R}$, which we also verify. This is the starting point for the next subinterval, with $(f_{i+1}, y_{i+1}) = p(t_{i+1})$.

Since hundreds of parameters are used in the justification of this step, we will not print them here, and just mention that the validator checks that the solution exists for all $t \leq 1$ for every $A \in [A_-, A_+]$, and that the propagated bounds obtained by this analysis yield \eqref{eq:positive-shot} and \eqref{eq:negative-shot}. See the Git repository \cite{Git-repo} for the complete collection of source and output files.

\smallskip
\subsubsection{Technical parameters}\label{sss:param}  
The table below lists the fixed technical parameters used in the validation and mentioned in the text above.

\begin{center}
\begin{tabular}{@{}cll@{}}
\toprule
$h_0$ & end of initial interval & $0.01$\\
$d_0$ & degree of $p_y$ on the initial interval & $40$ \\
$\overline{V}$ & coarse bound on the initial interval & $10^{-12}$ \\
$N_0$ & subdivisions on initial interval & $64$\\
$\overline{h}$ & maximum subsequence step size & $0.01$ \\
$d_1$ & degree of $p_y$ on subsequent steps & $30$\\
$\overline{R}$ & coarse bound on subsequent steps & $10^{-10}$ \\
$N_1$ & subdivisions on each subsequent step & $16$\\
& number of subsequent steps & $103$ per run, $309$ total\\
& working precision & $512$ bits \\
\bottomrule
\end{tabular}
\end{center}
{We refer the reader to the files stored in the Git repository \cite{Git-repo} for more details.}

\subsection{Shooting and odd extension}
The next step towards completing the proof of \cref{thm:counterexample-main} involves performing an odd reflection across the endpoint $s=0$ in order to obtain a solution on the entire $s$-interval $[-1,1]$. Let us begin by deducing from the previous facts the existence of a real-analytic solution shooting to zero. We henceforth return to our original variable $s$, since we no longer require the reparameterization via $t$.

\begin{lemma}
\label{lem:shooting}
There exists a real number $\widehat A\in(A_-,A_+)$ such that the associated real analytic solution $f_{\widehat{A}}$ satisfies $f_{\widehat A}(0)=0$.
    
\begin{proof}
By \cref{lem:endpoint-family}, given the interval $I = [A_-, A_+]$, there exists $\eps > 0$ such that $A \mapsto (f_A(1 - \eps), f_A'(1 - \eps))$ is continuous in $I$. Then by \cref{prop:validated-shooting}, the solutions to each of the regular initial value problems
\[
    \begin{cases}
        \displaystyle \bar f_A''(s) = \frac{9 s \bar f_A' - N(s, \bar f_A, \bar f_A') / D(s, \bar f_A, \bar f_A')}{9(1-s^2)} \\
        \bar f_A(1-\eps) = f_A(1-\eps), \\
        \bar f_A'(1-\eps) = f_A'(1-\eps)
    \end{cases}
\]
admit unique solutions on $[0, 1-\eps]$, and thus their endpoint values depend continuously on their initial data, which in turn is a continuous function of $A$. By uniqueness of solutions, this is the unique solution extending $f_A$ on $[0, 1]$, and in particular $A \mapsto \Phi(A) := f_A(0)$ is continuous. The conclusion now follows by \eqref{eq:positive-shot} and \eqref{eq:negative-shot}.
\end{proof}
\end{lemma}

We now use the solution $f = f_{\hat A}$ given by \cref{lem:shooting} to obtain a solution on $[-1,1]$. Observe that the system $(s,f,f')$ is invariant under the reflection map
\[
  (s,f,f')\longmapsto(-s,-f,f')\,.
\]
Indeed, $K$ and $T$ are conjugated, $D$ is unchanged, and $N$ changes sign. Thus, we may define $\widetilde f(s)=-f(-s)$ for $s\in [-1,0]$, which satisfies $\widetilde f(0)=f(0)=0$ and $\widetilde f'(0)=f'(0)$.  Since the ODE is nonsingular at $s=0$ and $D>0$ throughout, uniqueness gives the real-analytic odd extension
\begin{equation*}\label{eq:f-odd}
  f(s) = \begin{cases}
      f(s) & 0 \leq s \leq 1 \\
      \tilde f(s) & -1 \leq s < 0\,.
  \end{cases}
\end{equation*}

\subsection{Conclusion of the proof}
Now that we have obtained a solution of \eqref{eq:shooting-ODE} on $[-1,1]$ which in turn yields a real-analytic solution $u$ of \eqref{eq:special-lagrangian} away from the origin, we are in a position to conclude the proof of \cref{thm:counterexample-main}.

\begin{proof}[Proof of \cref{thm:counterexample-main}]
First notice that $\det(\Id + i D^2 u)$ never vanishes, because the matrix $\Id + i D^2 u$ is always invertible, as $(\Id - i D^2 u)(\Id + i D^2 u) = \Id + (D^2 u)^2$ is positive-definite. By \eqref{eq:special-lagrangian}, the argument of $\det(\Id + i D^2 u)$ takes values in $\pi \Zbb$, hence it is constant by continuity. Thus our argument condition \eqref{eq:argument-at-start} implies that it is zero, namely
\begin{equation}\label{eq:SL-off-origin}
  F_5(D^2u)
  =\sum_{j=1}^5 \arctan \lambda_j (D^2 u)
  = 0
  \qquad\text{on }\R^5\setminus\{0\}.
\end{equation}

Since the Hessian of $u$ is $0$-homogeneous, it is in particular bounded, hence $u \in C^{1,1}$, but it is not even twice differentiable at the origin: if it were so, being $2$-homogeneous it would agree with its quadratic Taylor polynomial, but $u$ being odd (because $f$ and $P$ are so) and nonzero (because $f(1) = \widehat A > 0$) it cannot be a quadratic form.

Equation \eqref{eq:SL-off-origin} holds almost everywhere.  The operator $F_5$ is continuous and nondecreasing (cf. \cref{ss:viscosity-setup}), so
\cref{lem:ae-viscosity} implies that $u$ is a global viscosity solution.
\end{proof}

\appendix

\section{\texorpdfstring{Explicit formulae for $D$ and $N$}{Explicit formulae for D and N}} \label{app:polynomials}
The functions in \eqref{eq:DN-def} have the following explicit polynomial expressions.
\begin{align*}
D={}&16f^4-72f^3sy+108f^2s^2y^2-108f^2y^2-24f^2\\
&-54fs^3y^3+54fsy^3+54fsy-27s^2y^2+27y^2+1,
\end{align*}
and
\begin{align*}
N={}&32f^5-144f^4sy+288f^3s^2y^2-288f^3y^2-80f^3\\
&-432f^2s^3y^3+432f^2sy^3+216f^2sy\\
&+486fs^4y^4-972fs^2y^4-216fs^2y^2+486fy^4+216fy^2+10f\\
&-243s^5y^5+486s^3y^5+108s^3y^3-243sy^5-108sy^3-9sy.
\end{align*}
These identities have been obtained by expanding \eqref{eq:DN-def}, with the aid of SymPy.

\printbibliography

\end{document}